\documentclass[11pt, reqno]{amsart}
\usepackage{amsmath,amsfonts,amssymb,amsthm,color,cancel}
\usepackage{epsfig}
\usepackage[normalem]{ulem}
\usepackage{ulem}

\usepackage{bbm}
\usepackage{enumerate}
\usepackage{amstext}
\usepackage{mathrsfs}
\usepackage{xspace}
\usepackage{amsfonts}
\usepackage{amsmath}
\usepackage{amssymb}
\usepackage{amstext}
\usepackage{amsthm}       %proof environment :)
\usepackage{xspace}
\usepackage[active]{srcltx}
\usepackage{cite}

\newtheorem{theorem}{Theorem}
\newtheorem{lemma}{Lemma}[section]
\newtheorem{definition}[lemma]{Definition}

\newtheorem{proposition}[lemma]{Proposition}
\newtheorem{remark}[lemma]{Remark}
\newtheorem{corollary}[lemma]{Corollary}

\newcommand{\cV}{\mathcal{V}}
\newcommand{\gm}{\gamma}
\newcommand{\M}{\mathcal{M}}
\newcommand{\R}{\mathbb{R}}
\newcommand{\cS}{\mathcal{S}}

\newcommand{\T}{\mathbb{T}}
\newcommand{\N}{\mathcal{N}}

\newcommand{\A}{\mathcal{A}}

\newcommand{\ri}{\rightarrow}
\newcommand{\eps}{\varepsilon}
\newcommand{\Proof}{\begin{proof}}
\newcommand{\End}{\end{proof}}
\newcommand{\EEnd}{\ensuremath{\hfill{\Box}}\\}

\numberwithin{equation}{section}

\begin{document}

%%%%% To ease editing, add:

%\baselineskip=15pt

%%%%%%%%%%%%%%%%

\title[Aubry-Mather theory III]{Aubry-Mather theory for contact Hamiltonian systems III}

\author{Panrui Ni \and Lin Wang}
\address{Shanghai Center for Mathematical Sciences, Fudan University, Shanghai 200433, China}

\email{prni18@fudan.edu.cn}
\address{School of Mathematics and Statistics, Beijing Institute of Technology, Beijing 100081, China}

\email{lwang@bit.edu.cn}
%\address{School of Mathematical Sciences, Fudan University, Shanghai 200433, China}
%
%\email{yanjun@fudan.edu.cn}

%\thanks{This work is partially supported by NSFC Grant No. 11790273, 12122109.
%}

%%%%%%%%
\begin{abstract}
\noindent By exploiting the { contact Hamiltonian dynamics $(T^*M\times\mathbb R,\Phi_t)$} around the Aubry set of contact Hamiltonian systems, we provide a relation among the Mather set, the {$\Phi_t$-}recurrent set, the strongly static set, the Aubry set, the Ma\~{n}\'{e} set and the {$\Phi_t$-}non-wandering set. Moreover, we consider the strongly static set, as a new flow-invariant set between the Mather set and the Aubry set, in the strictly increasing case. We show that this set plays an essential role in the representation of certain minimal forward weak KAM solution and the existence of  transitive orbits around the Aubry set.
\end{abstract}

\date{\today}
\maketitle

\tableofcontents

%\newpage
%%%%%%%%%%%%%%%%%%%%%%%%%%%%%%%%%%%%%%%%%%%%%%%%%%%%%%%%Sect. 1

\section{Introduction}
\setcounter{equation}{0}
In \cite{MS,WWY2}, the Aubry-Mather theory was developed for  conformally symplectic systems and contact Hamiltonian systems with strictly increasing dependence on the contact variable $u$ respectively.  The conformally symplectic systems are closely related to discounted Hamiltonian systems (see {\it e.g.}, \cite{Dav,MK}), which serve as a class of typical examples for more general contact cases.
In \cite{WWY4}, the Aubry-Mather theory was further developed for contact Hamiltonian systems with non-decreasing dependence on $u$. More information on the Aubry set was founded, such as the comparison property {(that is, given any viscosity solutions $u_-,v_-$ of (\ref{hj}) below, if $u_-\leq v_-$ on the projected Aubry set, then $u_-\leq v_-$ everywhere)}, graph property and
  {a partially ordered relation}  for the collection of all projected Aubry sets with respect to backward weak KAM solutions.
%  In strictly increasing cases, it was proved that the Ma\~{n}\'{e} set is the same as the Aubry set.
Loosely speaking, the Aubry-Mather theory and weak KAM theory are two kinds of parallel  ways  to describe the global minimizing dynamics of contact Hamiltonian systems. The former is concerned with ``orbits", while the {latter} focus on ``weak KAM solutions". This kind of solutions can be viewed as certain generalization of generating functions in Hamiltonian systems.
%These two theories turn out to be most powerful by combining with them.
One can also see \cite[Section 46]{Ar} for a vivid description on the connection between orbits and solutions of Hamilton-Jacobi equations.

\subsection{Basic assumptions}\label{basic}
Assume $M$ is a connected, closed (compact without boundary) and smooth Riemannian manifold. We choose, once and for all, a $C^\infty$ Riemannian metric $g$ on $M$. Denote by $\text{dist}(\cdot,\cdot)$ and  $d(\cdot,\cdot)$ the distance
 on $M$ and
  $T^*M\times\R$ induced by $g$ respectively. $C(M,\mathbb{R})$ stands for  the space of continuous functions on $M$. $\|\cdot\|_\infty$ denotes the supremum norm on $C(M,\mathbb{R})$.
$\|\cdot\|_x$ denotes a norm on $T^*_xM$ and $T_xM$. Let $H:T^*M\times\mathbb{R}\to \mathbb{R}$ be a $C^3$ function satisfying
\begin{itemize}
	\item [\textbf{(H1)}] {\it Strict convexity}:  $\frac{\partial^2 H}{\partial p^2} (x,p,u)$ is positive definite for all $(x,p,u)\in T^*M\times\R$;
	%\item [\textbf{(H2)}] \textbf{Superlinearity in the Fibers}: For every compact set $I$, $H(x,p,u)$ ($u\in I$) is uniformly superlinear growth with respect  to $p$;
	\item [\textbf{(H2)}] {\it Superlinearity}: for every $(x,u)\in M\times\R$, $H(x,p,u)$ is  superlinear in $p$;
	\item [\textbf{(H3)}] {\it Non-decreasing}: there is a constant $\lambda>0$ such that for every $(x,p,u)\in T^{\ast}M\times\R$,
		\begin{equation*}
		0\leq \frac{\partial H}{\partial u}(x,p,u)\leq \lambda.
		\end{equation*}
\end{itemize}
{Since results in this paper are based on the variational principle introduced in \cite{WWY}, which was proved under $C^3$ assumption for a technical reason, we assume $H(x,p,u)$ is of class $C^3$ here.} We consider the contact Hamiltonian system generated by
\begin{align}\label{c}\tag{CH}
\left\{
        \begin{array}{l}
        \dot{x}=\frac{\partial H}{\partial p}(x,p,u),\\
        \dot{p}=-\frac{\partial H}{\partial x}(x,p,u)-\frac{\partial H}{\partial u}(x,p,u)p,\qquad (x,p,u)\in T^*M\times\mathbb{R},\\
        \dot{u}=\frac{\partial H}{\partial p}(x,p,u)\cdot p-H(x,p,u).
         \end{array}
         \right.
\end{align}
{Denote the flow generated by (\ref{c}) by $\Phi_t$.} In order to handle global dynamics, it is necessary to  assume additionally
\begin{itemize}
\item [\textbf{(A)}] {\it Admissibility}:  there exists $a\in \R$ such that
\[\inf_{u\in C^\infty(M,\R)}\sup_{x\in M}H(x,Du,a)=0.\]
\end{itemize}
This formulation is inspired by the concept of the Ma\~{n}\'{e} critical value \cite{CIPP}.
From a PDE point of view, the assumption (A) holds true if and only if the stationary Hamilton-Jacobi equation
\begin{align}\label{hj}\tag{HJ}
H(x,Du,u)=0,\quad x\in M.
\end{align}
has a viscosity solution (see \cite[Theorem 1.4]{SWY}). If $H$ is independent of $u$, {according to \cite{LPV}, there is a unique constant $c_0$ called the critical value such that $H(x,Du)=c$ has viscosity solutions if and only if $c=c_0$. Here the admissibility holds if and only if $c_0$ is zero. If $c_0$ is not zero, we let $G:=H-c_0$. Then (A) always holds for the new Hamiltonian $G$.}

The necessity of (A) can be shown by the following example:
 \[H(x,p,u)=h(x,p)+g(x)u, \quad x\in \mathbb{T},\]
where { $\mathbb{T}=\mathbb R/\mathbb Z$} denotes a flat circle (once and for all). The function  $g:\mathbb{T}\to \R$ does not vanish identically and  satisfies $0\leq g(x)\leq \lambda$.
 If $g(x)>0$ for all $x\in \mathbb{T}$, based on the compactness of $\mathbb{T}$, $g(x)\geq \delta$ for certain positive constant $\delta$. In this case, {according to the comparison principle (see \cite[Theorem 3.2]{inc} for instance), $h(x,Du)+g(x)u=0$  has a unique viscosity solution.} Namely, (A) always holds.
If there exist $x_0\in \mathbb{T}$ such that  $g(x_0)=0$, then (A) may not hold. For example, consider the Hamilton-Jacobi equation
\[\frac{1}{2}|Du|^2+V(x)+g(x)u=0.\]
 Assume $V:\mathbb{T}\to\R$ is of class $C^3$ with $V(x_0)>0$ and $g(x_0)=0$. Then for all $a\in \R$,
\begin{align*}
&\inf_{u\in C^\infty(\mathbb{T},\R)}\sup_{x\in \mathbb{T}}\left\{\frac{1}{2}|Du|^2+V(x)+g(x)a\right\}\\
\geq &\inf_{u\in C^\infty(\mathbb{T},\R)}\left\{\frac{1}{2}|Du|^2+V(x_0)+g(x_0)a\right\}\\
= &V(x_0)>0.
\end{align*}
Therefore, (A) is necessary to be assumed.

\subsection{Aims, obstructions and contributions}

 In {\cite{Mn3}}, R. Ma\~{n}\'{e} obtained some properties of the Aubry set from the perspective of topological dynamics. Inspired by this work, we are concerned with the following problems.
\begin{itemize}
\item  The topological dynamics on the Aubry set, such as the recurrence property, the non-wandering property and their relations to the Mather set and Ma\~{n}\'{e} set. { In the following, we mean by ``recurrence property" and ``non-wandering property" the recurrence and non-wandering property with respect to the dynamics generated by contact Hamiltonian flow $\Phi_t$.}
\item  The representation of weak KAM solutions, and the {interplay} between weak KAM solutions and the dynamics  around the Aubry set.
\end{itemize}
{ In classical cases, i.e. ${\partial_u H}\equiv 0$, the backward and forward weak KAM solutions are one-to-one correspondent, which are called the conjugate pairs, see \cite{Fat-b}. Their sets have the same cardinality. Under (H1)-(H3) and (A), the structure of the set of the backward and the one of  forward weak KAM solutions are quite different from classical cases.} We have to deal with some new issues as follows.
\begin{itemize}
 \item  [(1)] In classical cases,  the Aubry set is chain-recurrent (see  \cite{CDI,Mn3}). Even in strictly increasing contact cases, the Aubry set may contain non-chain recurrent points (see Proposition \ref{exppp}(i)(ii) below).
    \item [(2)] In strictly increasing cases, the backward weak KAM solution is always unique. Unfortunately, the dynamics reflected by the backward weak KAM solution is too rough. Thus, one has to  exploit the structure of the set of forward weak KAM solutions to reveal more dynamical information. However, the  structure of this set is rather complicated (see Proposition \ref{exppp}(iii)(iv) below).
         \item [(3)]  The complicated structure of the set of weak KAM solutions causes  certain  difficulties to show the the {interplay} between weak KAM solutions and the dynamics  around the Aubry set. For example,  even if $\partial_u H$ vanishes at only one point,  some new phenomena from both dynamical and PDE aspects would appear. More precisely, we consider
\[\frac{1}{2}|Du|^2+f(x)u=0,\quad x\in \mathbb{T},\]
where  $f:\mathbb{T}\to\mathbb{R}$ is a $C^3$ function with $f(x_0)=0$, $f(x)>0$ for all $x\in \mathbb{T}\backslash\{x_0\}$. It is clear that $u\equiv 0$ is a viscosity solution. Besides,
 there exists an uncountable family of nontrivial  viscosity solutions $\{v_i\}_{i\in I}$ and the definition of the Aubry set essentially depends on $v_i$ (see \cite[Proposition 1.11]{WWY4} for more details).
    \end{itemize}

Corresponding to the issues above, we summarize the main contributions in this paper  as follows:
\begin{itemize}
 \item  Regarding Item (1), we find the Aubry set is too large to characterize the dynamics with recurrent or non-wandering property. Thus, we introduce so called the strongly static set, which is a new flow invariant subset of the Aubry set. We prove that the strongly static set is always non-wandering. Moreover, in order to locate this set in a series of action minimizing invariant sets, we prove an inclusion relation among the Mather set, the Aubry set, the Ma\~{n}\'{e} set, the recurrent points and the non-wandering points. {For the definitions of the first three sets, see Section \ref{auset} below. The latter two sets are the basic concepts in the classical theory of topological dynamical systems.} This result is given by Theorem \ref{raomeome}. It is worth mentioning that the strongly static set always coincides with the Aubry set in the classical case (see Remark \ref{asu} below).
     \item  Regarding  Item (2), we focus on the structure of the set of forward weak KAM solutions  in the strictly increasing case. The existence of the {\it maximal} element in this set was shown in \cite{WWY2}. Unfortunately, the {\it minimal} forward weak KAM solution may not exist in the sense of total ordering. For example, \begin{equation}\label{conterex}
{u}+\frac{1}{2}|D{ u}|^2=0,\quad x\in\mathbb{T}.
\end{equation}
Let $(-\frac 12,\frac 12]$  be a fundamental domain of $\mathbb{T}$. Let $\cS_+$ be the set of all forward weak KAM solutions of (\ref{conterex}). A direct calculation shows
\[v(x):=\min_{\cS_+}v_+(x)\equiv -\frac{1}{8},\]
which is not a forward weak KAM solution of (\ref{conterex}). As a substitute, we prove the existence of the {\it minimal} forward weak KAM solution in a partially ordered sense by virtue of Zorn's lemma. { In the classical case ${\partial_u H}\equiv 0$, both of the sets of backward and forward solutions are not bounded. Hence, it is not meaningful to introduce the maximal  or minimal solutions. On the other hand,  given a backward solution in the classical case, the forward one conjugated to this solution is unique. In this sense, the maximal solution is the same as the minimal one.} Moreover, we  show that the strongly static set plays an essential role in the representation of the {\it minimal} forward weak KAM solution. This result is given by Theorem \ref{pooo}. Loosely speaking, in the strictly increasing cases, the Aubry set is only related to the unique backward weak KAM solution and the maximal one. The strongly static set is necessarily involved in order to characterize the property of forward weak KAM solutions except the maximal one.
     \item   Regarding  Item (3), since the Aubry set may contain wandering points, we need to introduce a more flexible dynamics to detect the {interplay} between weak KAM solutions and the dynamics  around the Aubry set. The non-wandering property can be viewed as ``neighborhood recurrence".  Thus, we consider a kind of dynamical property that can be viewed as ``neighborhood transition". More precisely, we introduce the following definition.
\begin{definition}[{Transitive orbit}]\label{ccnnet} {Given $X_1, X_2\in T^*M\times\R$, we say there is a transitive orbit from $X_1$ to $X_2$ if for any neighborhoods $U_1$ of $X_1$ and  $U_2$ of $X_2$, there exists an orbit that begins in $U_1$ and later passes through $U_2$.}
\end{definition}
\begin{remark}
It is clear that a transitive orbit { from $X_1$ to $X_2$} is a { special} pseudo-orbit with arbitrarily small jumps. In particular, { there are at most two jumps of the transitive orbit, and these jumps are only allowed to happen around the adjoining points of $X_1$ and $X_2$. {In \cite[Definition 1.1.8]{PS}, the authors defined a relation $X_1\rightsquigarrow X_2$ if there is a pseudo-orbit from $X_1$ to $X_2$. In the following, we still write} $X_1\rightsquigarrow X_2$ if there is a transitive orbit from $X_1$ to $X_2$. }
\end{remark}

Finally, we obtain a result on the interplay among weak KAM solutions, the strongly static set and the existence of  transitive orbits around the Aubry set. This result is given by Theorem \ref{CONC}.
    \end{itemize}

\section{Statement of main results}

To state the main results (Theorem \ref{raomeome}, Theorem \ref{pooo} and Theorem \ref{CONC} below) in a precise way, we need to prepare some notions and notations. They mainly come from \cite{WWY,WWY1,WWY2,WWY3,WWY4}.

\subsection{Notions and notations}
\subsubsection{Weak KAM solutions}
Let $L:TM\times\R\to\R$ be the contact Lagrangian  associated to $H(x,p,u)$ via
\[
L(x,\dot{x},u):=\sup_{p\in T^*_xM}\left\{\langle \dot{x},p\rangle_x-H(x,p,u)\right\},
\]
where $\langle \cdot ,\cdot\rangle_x$ represents the canonical pairing between the tangent and cotangent space at $x\in M$. Since $H$ satisfies (H1), (H2) and (H3), then
 $L(x,\dot{x},u)$  satisfies
\begin{itemize}
\item [\textbf{(L1)}] {\it Strict convexity}:  $\frac{\partial^2 L}{\partial {\dot{x}}^2} (x,\dot{x},u)$ is positive definite for all $(x,\dot{x},u)\in TM\times\R$;
%\item [\textbf{(L2)}] \textbf{Superlinearity in the Fibers}: For every compact set $I$, $L(x,\dot{x},u)$ ($u\in I$) is uniformly superlinear growth with respect  to $\dot{x}$;
    \item [\textbf{(L2)}] {\it Superlinearity}: for every $(x,u)\in M\times\R$, $L(x,\dot{x},u)$ is superlinear in $\dot{x}$;
\item [\textbf{(L3)}] {\it Non-increasing}: there is a constant $\lambda>0$ such that for every $(x,\dot{x},u)\in TM\times\R$,
                        \begin{equation*}
                        -\lambda\leq \frac{\partial L}{\partial u}(x,\dot{x},u)\leq 0.
                        \end{equation*}
\end{itemize}

{Following Fathi \cite{Fat-b}, one can define weak KAM solutions of (\ref{hj}), see \cite[Definitions 2.1 and 2.2]{SWY}. According to \cite[Lemmas 4.1, 4.2, 6.2]{SWY}), the backward weak KAM solutions of (\ref{hj}) are equivalent to the viscosity solutions.}
\begin{definition}\label{bwkam}
	A function $u_-\in C(M,\mathbb{R})$ is called a backward weak KAM solution of \eqref{hj} if
	\begin{itemize}
		\item [(i)] for each continuous piecewise $C^1$ curve $\gamma:[t_1,t_2]\rightarrow M$, we have
		\[
		u_-(\gamma(t_2))-u_-(\gamma(t_1))\leq\int_{t_1}^{t_2}L(\gamma(s),\dot{\gamma}(s),u_-(\gamma(s)))ds;
		\]
		\item [(ii)] for each $x\in M$, there exists a $C^1$ curve $\gamma:(-\infty,0]\rightarrow M$ with $\gamma(0)=x$ such that
		\begin{align}\label{cali1}
		u_-(x)-u_-(\gamma(t))=\int^{0}_{t}L(\gamma(s),\dot{\gamma}(s),u_-(\gamma(s)))ds, \quad \forall t<0.
		\end{align}
	\end{itemize}
%By analogy with \cite{Fat-b}, (i) reads that ${u_-}$  is dominated by $L$, denoted by ${u_-}\prec L$. The curves in (ii) are called $({u_-},L,0)$-calibrated curves.

Similarly, 	a function $u_+\in C(M,\mathbb{R})$ is called a forward weak KAM solution of  \eqref{hj} if it satisfies (i) and
for each $x\in M$, there exists a $C^1$ curve $\gamma:[0,+\infty)\rightarrow M$ with $\gamma(0)=x$ such that
		\begin{align}\label{cali2}
		u_+(\gm(t))-u_+(x)=\int_{0}^{t}L(\gamma(s),\dot{\gamma}(s),u_+(\gamma(s)))ds,\quad \forall t>0.
		\end{align}
	We denote by $\mathcal{S}_-$ (resp. $\mathcal{S}_+$) the set of backward (resp. forward) weak KAM solutions of equation \eqref{hj}.
\end{definition}

%For the reader's convenience, we will recall some notions, notations, facts and tools introduced by \cite{SWY,WWY,WWY1,WWY2,WWY3,WWY4} in Appendix \ref{apA}.  More precisely, the following ones will be involved: the forward and backward action functions $h_{x_0,u_0}(x,t)$, $h^{x_0,u_0}(x,t)$; the Lax-Oleinik semigroups $T_t^{\pm}$; weak KAM solutions; the globally minimizing curves, positively and negatively semi-static curves, and so forth.

\subsubsection{Action minimizing objects}\label{auset}

The definitions of the action minimizing invariant sets  are based on the variational principle of contact Hamiltonian systems. See \cite[Theorem A]{WWY} for the following result, which holds under (H1), (H2) and $|\frac{\partial H}{\partial u}|\leq \lambda$ instead of (H3).
\begin{proposition}\label{IVP}
	For any given $x_0\in M$, $u_0\in\mathbb{R}$, there exists a  continuous function $h_{x_0,u_0}(x,t)$ defined on $M\times(0,+\infty)$
	satisfying
\begin{equation}\label{baacf1}
	h_{x_0,u_0}(x,t)=u_0+\inf_{\substack{\gamma(0)=x_0 \\  \gamma(t)=x} }\int_0^tL(\gamma(\tau),\dot{\gamma}(\tau),h_{x_0,u_0}(\gamma(\tau),\tau))d\tau,
\end{equation}
where the infimum is taken among  Lipschitz continuous curves $\gamma:[0,t]\rightarrow M$.
 Moreover, the infimum in (\ref{baacf1}) is achieved. Let $\gamma$ be a Lipschitz curve achieving the infimum and
	\[x(s):=\gamma(s),\ u(s):=h_{x_0,u_0}(\gamma(s),s),\  p(s):=\frac{\partial L}{\partial \dot{x}}(\gamma(s),\dot{\gamma}(s),u(s)).
	\]
	Then $(x(\cdot),p(\cdot),u(\cdot)):[0,t]\to T^*M\times\R$ satisfies equations (\ref{c}) with $x(0)=x_0$, $x(t)=x$ and \[\lim_{s\rightarrow 0^+}u(s)=u_0.\]
	\end{proposition}
 We associate to  $h_{x_0,u_0}(x,t)$ another action function $h^{x_0,u_0}(x,t)$, which is also  defined implicitly by
\begin{align}\label{2-3}
h^{x_0,u_0}(x,t)=u_0-\inf_{\substack{\gamma(t)=x_0 \\  \gamma(0)=x } }\int_0^tL(\gamma(\tau),\dot{\gamma}(\tau),h^{x_0,u_0}(\gamma(\tau),t-\tau))d\tau,
\end{align}
where the infimum is taken among Lipschitz continuous curves $\gamma:[0,t]\rightarrow M$.
%The definitions of the action functions are still valid if we replace (H3) by $|\partial_u H|\leq \lambda$.

%For more properties of the action functions, see Appendix \ref{A.2}.

Based on the action functions, one can define action minimizing curves, see \cite[Definition 3.1]{WWY2}.

%For more properties of this kind of curves, see Appendix \ref{A.5}.
\begin{definition}[Globally minimizing curves]\label{gm}
	 A curve $(x(\cdot),u(\cdot)):\mathbb{R}\to M\times\mathbb{R}$ is called globally minimizing, if it is locally Lipschitz and
		 for each $t_1$, $t_2\in\mathbb{R}$ with $t_1< t_2$, there holds
		\begin{align}\label{globalm}
		u(t_2)=h_{x(t_1),u(t_1)}(x(t_2),t_2-t_1).
		\end{align}
%A {\it global minimizing orbit} is defined as $(x(\cdot),p(\cdot),u(\cdot)):\R\ri T^*M\times\R$ by adding \[p(t):=\frac{\partial L}{\partial \dot{x}}(x(t),\dot{x}(t),u(t)),\quad \forall t\in \R.\]
\end{definition}
The positively minimizing curves (resp. negatively minimizing curves) can be defined in a similar manner. We say positively (resp. negatively), we mean the curve is defined on $\mathbb R_+$ (resp. $\mathbb R_-$), and (\ref{globalm}) holds for $t_1$, $t_2\in\mathbb R_+$ (resp. $\in\mathbb R_-$). If a curve $(x(\cdot),u(\cdot)):\mathbb{R}\to M\times\mathbb{R}$ is  global minimizing, then $x:\R\to M$ is of class $C^1$. Let
\begin{equation}\label{p}
p(t):=\frac{\partial L}{\partial \dot{x}}(x(t),\dot{x}(t),u(t)),\ \ t\in\R.
\end{equation}
Then $(x(\cdot),p(\cdot),u(\cdot)):\R\to T^*M\times\R$ satisfies equations (\ref{c}) (see \cite[Propostion 3.1]{WWY2}).
Following Ma\~{n}\'{e} \cite{Mn3},  the notion of static and semi-static curves for contact Hamiltonian systems were introduced in \cite[Defintion 3.2]{WWY2} and \cite[Definition 1.1]{WWY4} respectively.

\begin{definition}[Semi-static curves]\label{semdepp}
A curve $(x(\cdot),u(\cdot)):\mathbb{R}\to M\times\mathbb{R}$ is called semi-static, if it is globally minimizing and for each $t_1\leq t_2\in\mathbb{R}$, there holds
	\begin{equation}\label{3-3}
	u(t_2)=\inf_{s>0}h_{x(t_1),u(t_1)}(x(t_2),s).
	\end{equation}
\end{definition}

\begin{definition}[Semi-static orbits]\label{semorbit}
If a curve $(x(\cdot),u(\cdot)):\mathbb{R}\to M\times\mathbb{R}$ is  semi-static , then $(x(\cdot),p(\cdot),u(\cdot)):\R\to T^*M\times\R$ satisfies equations (\ref{c}), where $p(\cdot)$ is determined by (\ref{p}). We call it  a  semi-static  orbit.
\end{definition}
The positively  (resp. negatively) semi-static orbits can be also defined in a similar manner. {Recall that the flow generated by (\ref{c}) is $\Phi_t$. We define some $\Phi_t$-invariant sets as follows.}

%\begin{itemize}
%\item the Ma\~{n}\'{e} set $\tilde{\mathcal N}$, see Definition \ref{audeine} below.
%
%\item the Aubry set $\tilde{\mathcal A}$, see Definition \ref{audeine99} below.
%
%\item the Mather set $\tilde{\mathcal M}$, see Equation (\ref{matherset}) below.
%\end{itemize}
%Similar to the classical Hamiltonian systems, it has been proved that for contact Hamiltonian systems, we have the inclusion relation \[\emptyset\neq \tilde{\mathcal M}\subseteq\tilde{\mathcal A}\subseteq \tilde{\mathcal N},\]
%see \cite[Theorem 1.5(3)]{WWY4}.

\begin{definition}[Ma\~{n}\'{e} set]\label{audeine}
 We call the set of all semi-static orbits the Ma\~{n}\'{e} set for $H$, denoted by $\tilde{\mathcal{N}}$.
\end{definition}
%By  definition, the Ma\~{n}\'{e} set is an invariant subset of $T^*M\times\mathbb{R}$ by $\Phi_t$.
We call $\mathcal{N}:=\pi^*\tilde{\mathcal{N}}$ the projected Ma\~{n}\'{e} set. We denote, once and for all
\[ \pi^*:T^*M\times\R\rightarrow M.\]
 We define $\tilde{\mathcal{N}}^+$ (resp. $\tilde{\mathcal{N}}^-$) as the set of all positively (resp. negatively) semi-static orbits.

\begin{definition}[Static curves]\label{stcont}
	A curve $(x(\cdot),u(\cdot)):\mathbb{R}\to M\times\mathbb{R}$ is called static, if it is globally minimizing and for each $t_1, t_2\in\mathbb{R}$, there holds
	\begin{equation}\label{3-399}
	u(t_2)=\inf_{s>0}h_{x(t_1),u(t_1)}(x(t_2),s)
\end{equation}
\end{definition}
\noindent A {\it static orbit} is defined as $(x(\cdot),p(\cdot),u(\cdot)):\R\ri T^*M\times\R$, where $p(\cdot)$ is determined by (\ref{p}).
%where $L:TM\times\R\to\R$ denotes the contact Lagrangian  associated to $H(x,p,u)$ via
%\[
%L(x,\dot{x},u):=\sup_{p\in T^*_xM}\left\{\langle \dot{x},p\rangle_x-H(x,p,u)\right\}.
%\]

\begin{definition}[Aubry set]\label{audeine99}We call the set of all static orbits  the Aubry set for $H$, denoted by $\tilde{\mathcal{A}}$. The Aubry set is also called the static set.
\end{definition}
We call $\mathcal{A}:=\pi^*\tilde{\mathcal{A}}$ the projected Aubry set.
Following Mather \cite{M1}, we define a subset of the Aubry set from a measure theoretic point of view, so called the Mather set.
%Let $\cS_-$  be the set of the backward  weak KAM solutions of (\ref{hj}).
Based on Proposition \ref{88996} below,
% For each $v\in \cS_-$, $\tilde{\mathcal{N}}_{v}$ is a flow invariant subset of $T^*M\times\mathbb{R}$.
%  The set $\tilde{\mathcal{N}}_{v}$ is non-empty and compact (see Proposition \ref{88996} below). Consequently,
there exist Borel $\Phi_t$-invariant probability measures supported in $\tilde{\mathcal{N}}$, called {\it Mather measures}. Denote by $\mathfrak{M}$ the set of Mather measures. The {\it Mather set} of contact Hamiltonian systems (\ref{c}) is defined by
\begin{equation}\label{matherset}
  \tilde{\mathcal{M}}=\mathrm{cl}\left(\bigcup_{\mu\in \mathfrak{M}}\text{supp}(\mu)\right),
\end{equation}
where $\text{supp}(\mu)$ denotes the support of $\mu$ { and $\mathrm{cl}(A)$ denotes the closure of $A\subseteq T^*M\times \R$.}

The invariance of these sets above follows directly from their definitions.

\vspace{1ex}
\subsubsection{Strongly static set}\label{disss}

If $H$ is independent of $u$,  the Aubry set is chain-recurrent. Unfortunately, it is not true in general contact settings. In order to characterize the chain-recurrence in the Aubry set, we introduce a new flow invariant set, called strongly static set.
\begin{definition}[Strongly static curves]\label{stcontx}
	A curve $(x(\cdot),u(\cdot)):\mathbb{R}\to M\times\mathbb{R}$ is called strongly static, if it is globally minimizing and for each $t_1, t_2\in\mathbb{R}$, there holds
	\begin{equation}\label{3-399x}
	u(t_2)=\sup_{s>0}h^{x(t_1),u(t_1)}(x(t_2),s).
\end{equation}
\end{definition}
A {\it strongly static orbit} is defined as $(x(\cdot),p(\cdot),u(\cdot)):\R\ri T^*M\times\R$, where $p(\cdot)$ is determined by (\ref{p}).
\begin{definition}[Strongly static set]\label{stcontxss}
	\[ \tilde{\mathcal{S}}_s:=\mathrm{cl}(\{\text{all\  strongly\  static\  orbits}\}). \]
%where $\mathrm{cl}(A)$ denotes the closure of $A\subseteq T^*M\times \R$.
\end{definition}
\begin{remark}\label{asu}
If $H(x,p,u)$ is independent of $u$, then
\[h_{x_0,u_0}(x,t)=u_0+h^t(x_0,x),\quad h^{x_0,u_0}(x,t)=u_0-h^t(x,x_0).\]
By definition, we have the Aubry set is the same as the strongly static set. {Loosely speaking, the set of strongly static curves is closely related to the set of forward weak KAM solutions $\cS_+$. In the contact setting under the assumptions (H1)-(H3), $\mathcal{S}_+$ is much more complicated than $\cS_-$ in general sense. Then it is natural to expect that  the strongly static set contains more specific dynamical properties compared to the static one.}
\end{remark}

In general cases, the differences between $\tilde{\A}$ and $\tilde{\mathcal{S}}_s$ are shown by the following Proposition \ref{exppp}. For the consistency, we postpone its proof in \ref{proofexppp}.
\begin{proposition}\label{exppp}
Let $\lambda>0$ and
\begin{equation}\label{exhxxx}\tag{E}
H(x,p,u):=\lambda u+\frac{1}{2}|p|^2+p\cdot V(x),\quad x\in T,
\end{equation}
where $\mathbb{T}$ denotes a flat circle and $V:\mathbb{T}\ri \R$ is a $C^3$ function which has exactly two vanishing points $x_1$, $x_2$ with $V'(x_1)>0$, $V'(x_2)<0$. Let $\cS_-$  and $\cS_+$ be the set of the backward  and forward weak KAM solutions of (\ref{exhxxx}) respectively. Then $u_-\equiv 0$ is the unique element in $\cS_-$. Thus, $u_+\equiv 0\in \cS_+$. Moreover,
\begin{itemize}
\item [(i)] if $\lambda>|V'(x_2)|$, then the point $(x_2,0,0)$ is a sink  in $T^*\mathbb{T}\times\R$;
\item [(ii)] for any $\lambda>0$, \[\tilde{\A}=\left\{(x,0,0)\ |\ x\in \mathbb{T}\right\},\quad \tilde{\mathcal{S}}_s=\left\{(x_1,0,0),\ (x_2,0,0)\right\};\]
%\item [(iii)] for any $\lambda>0$, there exists another  $v_+\in\cS_+$ with $v_+(x_1)=0$, $v_+(x)<0$ for each $x\in \mathbb{T}\backslash \{x_1\}$;
    \item [(iii)] if $\lambda<|V'(x_2)|$, the set $\cS_+$ consists of  two elements $u_+\equiv 0$ and $w_+$, where $w_+:\mathbb{T}\to\R$ satisfies $w_+(x_1)=0$, $w_+(x)<0$ for each $x\in \mathbb{T}\backslash \{x_1\}$;
\item [(iv)] for $\lambda$ large enough, $\cS_+$ may contains more than two elements.
\end{itemize}
\end{proposition}
By Proposition \ref{exppp}(i)(ii),  $\tilde{\A}$ contains non-chain recurrent points in the example (\ref{exhxxx}). Nevertheless, $\tilde{\mathcal{S}}_s$ is non-wandering. For Item (iii), we have a rough picture for  $\cS_{\pm}$ with $V(x)=\sin x$ (see Figure \ref{fig2}).
\begin{figure}[htbp]
\small \centering
\includegraphics[width=7cm]{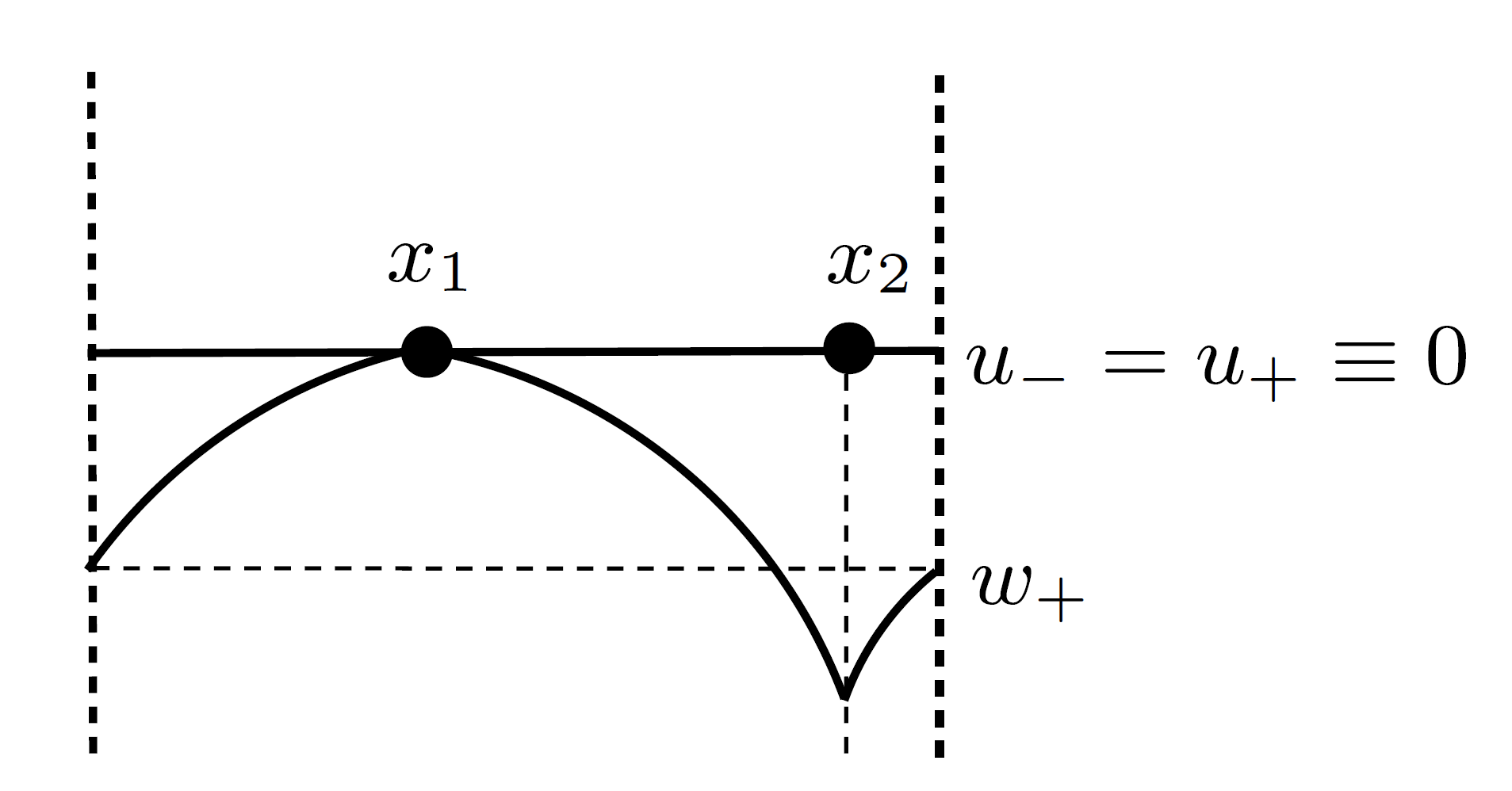}
\caption{$\cS_{\pm}$ in Item (iii)}\label{fig2}
\end{figure}

%\begin{remark}\label{17}
%  In \cite[Theorem 1.10]{WWY4}, it was shown that
%$\tilde{\mathcal{S}}_s\subseteq \tilde{\mathcal{A}}$ under the assumption (H1)-(H3) and (A). Especially, $\tilde{\mathcal{S}}_s=\tilde{\mathcal{A}}$ in classical cases. In view of Example (\ref{exhxxx}),  $\tilde{\mathcal{S}}_s$  may be  a proper subset of $ \tilde{\mathcal{A}}$ in the contact case.
%\end{remark}
% We will continually  return to this example to illustrate the necessity of the conditions required in certain results (see Remarks \ref{17}, \ref{necc} and \ref{anottr} below).

%Based on Proposition \ref{exppp}(i) and (ii), it seems that the Aubry set in Definition \ref{audeine99} is too large to characterize the recurrent or non-wandering  properties of the orbits in contact cases.
%We will show  the strongly static set is always non-wandering (see Theorem \ref{raomeome} below).
%

\subsection{Main results}\label{dissspre}
First of all, we  locate the strongly static set in a series of action minimizing invariant sets, and show its relations to the recurrence property and the non-wandering property.
\begin{theorem}[\bf Topological dynamics around the Aubry set]\label{raomeome}Let $\tilde{\mathcal{R}}$ be the set of recurrent points.  Let $\tilde{\Omega}$ be the set of non-wandering points. Then
\[\emptyset\neq\tilde{\mathcal{M}}\subseteq\tilde{\N}\cap \mathrm{cl}(\tilde{\mathcal{R}})\subseteq\tilde{\mathcal{S}}_s\subseteq \tilde{\A}\cap\tilde{\Omega}.\]
\end{theorem}

The closure $\text{cl}(\tilde{\mathcal{R}})$ is called the Birkhoff center. The fact $\tilde{\mathcal{M}}\subseteq\tilde{\N}\cap \mathrm{cl}(\tilde{\mathcal{R}})$ follows easily from the Poincar\'{e} recurrence theorem. To prove $\tilde{\N}\cap \mathrm{cl}(\tilde{\mathcal{R}})\subseteq\tilde{\mathcal{S}}_s$, we need to establish the Lipschitz continuity of {the Ma\~{n}\'{e} } potentials, whose proof is postponed to \ref{apB}. The inclusion $\tilde{\mathcal{S}}_s\subseteq \tilde{\A}\cap\tilde{\Omega}$ follows from a technical lemma on transitive criterion (see Lemma \ref{CONC1} below).

\vspace{1em}

In order to show the differences of the dynamics between the classical cases and the contact cases,  we enhance the assumption (H3) by
\begin{itemize}
	\item [\textbf{(H3')}] {\it Strictly increasing}: there is a constant $\lambda>0$ such that for every $(x,p,u)\in T^{\ast}M\times\R$,
		\begin{equation*}
		0< \frac{\partial H}{\partial u}(x,p,u)\leq \lambda,
		\end{equation*}
\end{itemize}
Under (H1), (H2), (H3') and (A), it is well known that the set of backward weak KAM solutions $\cS_-$ consists of only one element. {Consequently, the Ma\~{n}\'{e} set coincides with the Aubry set (see \cite[Remark 2]{WWY4}).} However, the structure of the set of forward weak KAM solutions $\cS_+$ may be rather complicated. {That causes significant differences  between the Aubry set and the strongly static set, as it is shown in Proposition \ref{exppp}.}

%In \cite[Theorem 1]{WWY3}, a rough description on the the structure of $\cS_+$ was given as follows.
%\begin{itemize}
%\item
%All of elements in $\mathcal S_+$ are uniformly bounded and equi-Lipschitz continuous.
%\item There exists the maximal element in $\mathcal S_+$, which is given by $u_+:=\lim_{t\to+\infty}T_t^+u_-$.
%\end{itemize}

In order to deal with the other elements in $\mathcal S_+$ except the maximal one.
We define a partial ordering  in $\cS_+$:

 \centerline{$v_1\preceq v_2$ if and only if $v_1(x)\leq v_2(x)$ for all $x\in M$.}

 \vspace{1ex}
Moreover, we define $\mathcal{Z}_{\max}$ a maximal totally ordered subset of $\cS_+$. Namely, for any $w_+\in \cS_+\backslash \mathcal{Z}_{\max}$, there exist $v_+\in \mathcal{Z}_{\max}$ and $x_1,x_2\in M$ such that $w_+(x_2)>v_+(x_2)$ and $w_+(x_2)<v_+(x_2)$.
We will show the existence of minimal elements in $\cS_+$ in this sense of partial ordering. Moreover, we will provide a representation for the minimal element in $\mathcal{Z}_{\max}$.

\begin{theorem}[\bf Minimal forward weak KAM solutions]\label{pooo}
\
\begin{itemize}
\item [(1)]
The  partially ordered set $(\mathcal{S}_+,\preceq)$ has minimal elements.
\item [(2)] For each $\mathcal{Z}_{\max}\subseteq \mathcal{S}_+$, there exists $x_0\in \mathcal{M}$ depending on $\mathcal{Z}_{\max}$, such that the minimal element $u^*$ in $\mathcal{Z}_{\max}$ can be represented in the following two manners:
\[u^*(x)=\inf_{\substack{v_+(x_0)=u_-(x_0)\\  v_+\in \cS_+ } }v_+(x)=\limsup_{t\to +\infty}h^{x_0,u_-(x_0)}(x,t),\]
where  $\mathcal{M}$ denotes the projected Mather set and $h^{\cdot,\cdot}(\cdot,\cdot): M\times\R\times M\times\R_+\to \R$  is the action function defined by (\ref{2-3}).
\end{itemize}
\end{theorem}

See Remark \ref{anottr} below for a discussion on the choice of $x_0$, from which one can see that the strongly static set plays an essential role.
{By analysing more in details the structure of  $\cS_+$,}  one has
\begin{theorem}[\bf Existence of transitive orbits]\label{CONC}
Given $X_1:=(x_1,p_1,u_1)\in \tilde{\A}$, $X_2:=(x_2,p_2,u_2)\in \tilde{\cS}_s$, if for each $v_+\in \cS_+$, $v_+(x_2)=u_-(x_2)$ implies $v_+(x_1)=u_-(x_1)$, then {$X_1\rightsquigarrow X_2$. }
\end{theorem}

 { If the forward weak KAM solution is unique,  $\tilde{\A}=\tilde{\cS}_s$ (see \cite[Proposition 10]{WWY4}). In this case, the projected strongly static set can be characterized as follows
 \[\cS_s=\{x\in M\ |\ u_-(x)=u_+(x)\},\]
 where $u_-$ (resp. $u_+$) denotes the unique backward (resp. forward) weak KAM solution. From Theorem \ref{CONC}, we have the following corollary.
\begin{corollary}\label{CONCtt}
 Given any two points $X_1,X_2\in\tilde{\A}$,  if the forward weak KAM solution is unique, then {$X_1\rightsquigarrow X_2$. }
\end{corollary}

%{
%\begin{remark}
%In \cite[Proposition 10]{dw}, it was shown that the uniqueness of the forward weak KAM solution can be implied by the hyperbolicity of fixed points.
%\end{remark}
%}

Figure \ref{fig} provides a rough picture for the dynamics  around $\tilde{\A}$ (projected to $M\times\R$) under the assumption of Theorem \ref{CONC}, where
\begin{itemize}
\item
$\Gamma_t$ denotes a transitive orbit from $X_1$ to $X_2$;
\item $\Gamma_n^1$, $\Gamma_n^2$ denote  non-wandering orbits to $X_1$ and $X_2$ respectively.
\end{itemize}

\begin{figure}[htbp]
\small \centering
\includegraphics[width=10cm]{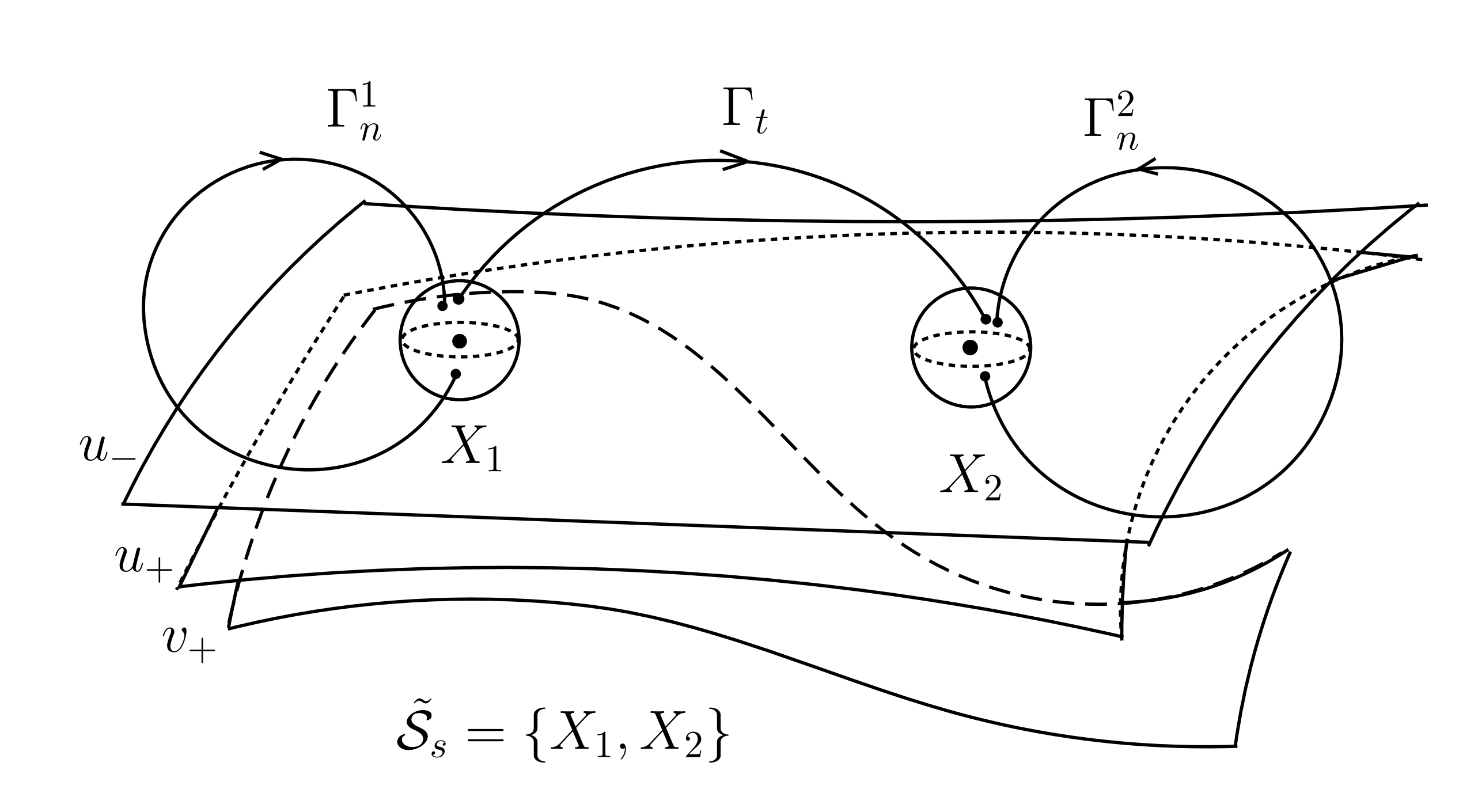}
\caption{The dynamics around the Aubry set}\label{fig}
\end{figure}

\vspace{1em}

The rest of this paper is organized as follows.  In Section \ref{apA}, we recall some useful facts, which mainly come from  \cite{WWY2,WWY4}.  In Section \ref{cll}, we prove a technical lemma on the existence of certain transitive orbit, from which we obtain the topological dynamics around the Aubry set in Section \ref{se44}. In Section  \ref{se55}, we provide more detailed information on the set of forward weak KAM solutions in the strictly increasing cases.  Some auxiliary results are proved in \ref{apC}. Finally,  we  provide a proof of Proposition \ref{exppp} in \ref{proofexppp}.

\vspace{1ex}

\section{Preliminaries}\label{apA}

Different from classical cases,  all of technical tools for contact Hamiltonian systems were formed in an implicit manner. The main reason to use the implicit form is to get rid of the constraints caused by the $u$-argument. Besides, it is worth mentioning that  an alternative variational formulation  was provided in \cite{CCJWY,LTW}  in light of G. Herglotz's work  \cite{her}.

In the following, we collect some facts used in this paper. All of these  results hold under (H1), (H2) and $|\partial_uH|\leq \lambda$.
\subsection{Action functions and minimizing curves}\label{A.2}
% \subsubsection{Forward action function}

%\begin{remark}
%The authors of \cite{CCWY}  formulated Theorem \ref{IVP} in a different way from \cite{WWY} in the spirit of Herglotz' generalized variational principle.
%The minimizer $x(t)$ of a minimization problem, together with $u(t)$ which is uniquely determined by Carath\'eodory equation
%\[
%\dot{u}(t) = L(x(t),u(t),\dot{x}(t))
%\]
%with initial conditions, satisfies
%the following generalized Euler-Lagrange equation  \[
%\frac{d}{dt}\frac{\partial L}{\partial \dot{x}}(x(t),u(t),\dot{x}(t))=\frac{\partial L}{\partial x}(x(t),u(t),\dot{x}(t))+\frac{\partial L}{\partial u}(x(t),u(t),\dot{x}(t))\cdot \frac{\partial L}{\partial \dot{x}}(x(t),u(t),\dot{x}(t)).
%\]	
%Let $p(t)=\frac{\partial L}{\partial \dot{x}}(x(t),u(t),\dot{x}(t))$. Then $(x(t),u(t),p(t))$ is a solution of \eqref{c}. See \cite{CCWY} for more details in this direction.
%\end{remark}

Let us collect some properties of the action functions $h_{x_0,u_0}(x,t)$ and $h^{x_0,u_0}(x,t)$ in the following propositions. See \cite[Theorems C, D]{WWY} and \cite[Theorem 3.1, Propositions 3.1-3.4]{WWY1} for more details.
\begin{proposition}\label{pri}\
\begin{itemize}
		\item [(1)] (Monotonicity).
		Given $x_0\in M$, $u_0$, $u_1$, $u_2\in\mathbb{R}$,
			if $u_1<u_2$, then $h_{x_0,u_1}(x,t)<h_{x_0,u_2}(x,t)$, for all $(x,t)\in M\times (0,+\infty)$;
		\item [(2)] (Minimality).
		Given $x_0$, $x\in M$, $u_0\in\mathbb{R}$ and $t>0$, let
		$S^{x,t}_{x_0,u_0}$ be the set of the solutions $(x(s),p(s),u(s))$ of (\ref{c}) on $[0,t]$ with $x(0)=x_0$, $x(t)=x$, $u(0)=u_0$.
		Then
		\[
		h_{x_0,u_0}(x,t)=\inf\{u(t)\ |\ (x(s),p(s),u(s))\in S^{x,t}_{x_0,u_0}\}, \quad \forall (x,t)\in M\times(0,+\infty).
		\]
		\item [(3)] (Lipschitz continuity).
		The function $(x_0,u_0,x,t)\mapsto h_{x_0,u_0}(x,t)$ is locally Lipschitz  continuous on $M\times\R\times M\times(0,+\infty)$.
		\item [(4)] (Markov property).
		Given $x_0\in M$, $u_0\in\mathbb{R}$,
		\[
		h_{x_0,u_0}(x,t+s)=\inf_{y\in M}h_{y,h_{x_0,u_0}(y,t)}(x,s)
		\]
		for all  $s$, $t>0$ and all $x\in M$. Moreover, the infimum is attained at $y$ if and only if there exists a minimizer $\gamma$ of $h_{x_0,u_0}(x,t+s)$ with $\gamma(t)=y$.
		\item [(5)] (Reversibility).
		Given $x_0$, $x\in M$ and $t>0$, for each $u\in \mathbb{R}$, there exists a unique $u_0\in \mathbb{R}$ such that
		\[
		h_{x_0,u_0}(x,t)=u.
		\]
	\end{itemize}
	\end{proposition}

%\subsubsection{Backward action function}
%
%
%See \cite[Theorem 3.1]{WWY1} for the following two results about backward  action functions.
%\begin{thm}\label{IMP-b}
%For any given $x_0\in M$ and $u_0\in\mathbb{R}$, there exists a  continuous function $h^{x_0,u_0}(x,t)$ defined on $M\times(0,+\infty)$
%satisfying (\ref{2-3}). Moreover, the infimum in (\ref{2-3}) is achieved. Let $\gamma$ be a curve achieving the infimum, and $x(s):=\gamma(s)$, $u(s):=h^{x_0,u_0}(x(s),t-s)$, $p(s):=\frac{\partial L}{\partial \dot{x}}(x(s),\dot{x}(s),u(s))$.
%Then $(x(s),p(s),u(s))$ satisfies equations (\ref{c}) with $x(0)=x$, $x(t)=x_0$ and $\lim_{s\rightarrow t^-}u(s)=u_0$.
%\end{thm}

%Furthermore, $h^{x_0,u_0}(x,t)$ has the following properties.
\begin{proposition}\label{pri1}\
\begin{itemize}
\item [(1)](Monotonicity).
Given $x_0\in M$ and $u_1$, $u_2\in\mathbb{R}$,
if $u_1<u_2$, then $h^{x_0,u_1}(x,t)<h^{x_0,u_2}(x,t)$, for all $(x,t)\in M\times (0,+\infty)$;
\item [(2)] (Maximality).
Given $x_0$, $x\in M$, $u_0\in\mathbb{R}$ and $t>0$, let
$S_{x,t}^{x_0,u_0}$ be the set of the solutions $(x(s),p(s),u(s))$ of  (\ref{c}) on $[0,t]$ with $x(0)=x$, $x(t)=x_0$, $u(t)=u_0$.
Then
\[
h^{x_0,u_0}(x,t)=\sup\{u(0)\ |\ (x(s),p(s),u(s))\in S_{x,t}^{x_0,u_0}\}, \quad \forall (x,t)\in M\times(0,+\infty).
\]
\item [(3)] (Lipschitz continuity).
The function $(x_0,u_0,x,t)\mapsto h^{x_0,u_0}(x,t)$ is locally Lipschitz  continuous on $M\times\R\times M\times(0,+\infty)$.
\item [(4)] (Markov property).
Given $x_0\in M$, $u_0\in\mathbb{R}$,
\[
h^{x_0,u_0}(x,t+s)=\sup_{y\in M}h^{y,h^{x_0,u_0}(y,t)}(x,s)
\]
for all  $s$, $t>0$ and all $x\in M$. Moreover, the supremum is attained at $y$ if and only if there exists a minimizer $\gamma$ of $h^{x_0,u_0}(x,t+s)$, such that $\gamma(t)=y$.
\item [(5)] (Reversibility).
Given $x_0$, $x\in M$, and $t>0$, for each $u\in \mathbb{R}$, there exists a unique $u_0\in \mathbb{R}$ such that
\[
h^{x_0,u_0}(x,t)=u.
\]
\end{itemize}
	\end{proposition}

%The relation between $h_{x_0,u_0}(x,t)$ and $h^{x_0,u_0}(x,t)$ is stated as follows. See \cite[Proposition 3.5]{WWY1} for details.
%
%
%\begin{proposition}\label{relation1} Given $x_0$, $x\in M$, $u_0$, $u\in\mathbb{R}$ and $t>0$, then $h_{x_0,u_0}(x,t)=u$ if and only if $h^{x,u}(x_0,t)=u_0$.
%\end{proposition}

By definition, we have
\begin{proposition}\label{pr11}
	Let $(x(\cdot),u(\cdot)):\mathbb{R}\to M\times\mathbb{R}$ be a globally minimizing curve. Then for all $t_1$, $t_2\in\mathbb{R}$ with $t_1\leq t_2$, \[u(t_2)=\inf_{s>0}h_{x(t_1),u(t_1)}(x(t_2),s)\text{\ \ if\  and\  only\  if\ \ } u(t_1)=\sup_{s>0}h^{x(t_2),u(t_2)}(x(t_1),s).\]
\end{proposition}

\begin{remark}\label{rcpps}
By Proposition \ref{pr11},
\begin{itemize}
\item
	a curve $(x(\cdot),u(\cdot)):\mathbb{R}\to M\times\mathbb{R}$ is globally minimizing if and only if for each $t_1< t_2\in\mathbb{R}$,
	\begin{equation}\label{3-3o}
	u(t_1)=h^{x(t_2),u(t_2)}(x(t_1),t_2-t_1);
	\end{equation}
	\item a curve $(x(\cdot),u(\cdot)):\mathbb{R}\to M\times\mathbb{R}$ is  semi-static if and only if  it is globally minimizing and for  each $t_1\leq t_2\in\mathbb{R}$,
\begin{equation}\label{3-3oo}
	u(t_1)=\sup_{s>0}h^{x(t_2),u(t_2)}(x(t_1),s);
	\end{equation}
\item a positively (resp. negatively) semi-static  curve can be also characterized in a similar manner.
\end{itemize}
\end{remark}

\subsection{Lax-Oleinik semigroups, weak KAM solutions and the Ma\~{n}\'{e} set}
Let us recall two  semigroups of operators introduced in \cite{WWY1}.  Define a family of nonlinear operators $\{T^-_t\}_{t\geq 0}$ from $C(M,\mathbb{R})$ to itself as follows. For each $\varphi\in C(M,\mathbb{R})$, denote by $(x,t)\mapsto T^-_t\varphi(x)$ the unique continuous function on $M\times[0,+\infty)$ such that

\[
T^-_t\varphi(x)=\inf_{\gamma}\left\{\varphi(\gamma(0))+\int_0^tL(\gamma(\tau),\dot{\gamma}(\tau),T^-_\tau\varphi(\gamma(\tau)))d\tau\right\},
\]
where the infimum is taken among absolutely continuous curves $\gamma:[0,t]\to M$ with $\gamma(t)=x$.  Let $\gamma$ be a curve achieving the infimum, and $x(s):=\gamma(s)$, $u(s):=T_s^-\varphi(x(s))$, $p(s):=\frac{\partial L}{\partial \dot{x}}(x(s),\dot{x}(s),u(s))$.
Then $(x(s),p(s),u(s))$ satisfies  (\ref{c}) with $x(t)=x$.

It is not difficult to see that $\{T^-_t\}_{t\geq 0}$ is a semigroup of operators and $T^-_t\varphi(x)$ is a viscosity solution of $w_t+H(x,w,w_x)=0$ with $w(x,0)=\varphi(x)$.

Similarly, one can define another semigroup of operators $\{T^+_t\}_{t\geq 0}$ by
\begin{equation*}\label{fixufor}
T^+_t\varphi(x)=\sup_{\gamma}\left\{\varphi(\gamma(t))-\int_0^tL(\gamma(\tau),\dot{\gamma}(\tau),T^+_{t-\tau}\varphi(\gamma(\tau)))d\tau\right\},
\end{equation*}
where the infimum is taken among absolutely  continuous curves $\gamma:[0,t]\to M$ with $\gamma(0)=x$. Let $\gamma$ be a curve achieving the infimum, and $x(s):=\gamma(s)$, $u(s):=T_{t-s}^+\varphi(x(s))$, $p(s):=\frac{\partial L}{\partial \dot{x}}(x(s),\dot{x}(s),u(s))$.
Then $(x(s),p(s),u(s))$ satisfies  (\ref{c}) with $x(0)=x$.

The following proposition gives a relation between Lax-Oleinik semigroups and action functions. See \cite[Propositions 4.1, 4.2]{WWY1} for details.
\begin{proposition}\label{pr4.2}
	For each  $\varphi\in C(M,\mathbb{R})$, we have
	\begin{equation*}
	T^{-}_t\varphi(x)=\inf_{y\in M}h_{y,\varphi(y)}(x,t),\quad
	T^{+}_t\varphi(x)=\sup_{y\in M}h^{y,\varphi(y)}(x,t),
	\quad  \forall (x,t)\in M\times(0,+\infty).
	\end{equation*}
\end{proposition}

The following proposition gives a relation between Lax-Oleinik semigroups and weak KAM solutions. See \cite[Lemmas 4.1, 4.2, 6.2]{SWY} for details.
\begin{proposition}\label{pr4.5}
The backward weak KAM solutions of (\ref{hj}) are the same as the viscosity solutions of (\ref{hj}). Moreover,
\begin{itemize}
\item
 [(i)] A function $u:M\to\R$ is a backward weak KAM solution of (\ref{hj}) if and only if $T^-_tu=u$ for all $t\geq 0$;
 \item [(ii)] A function $v:M\to\R$ is a forward weak KAM solution of (\ref{hj}) if and only if $T^+_tv=v$ for all $t\geq 0$.
 \end{itemize}
\end{proposition}

By \cite[Theorem 1.2]{WWY2} and \cite[Theorem 1]{WWY3}, we have

\begin{proposition}\label{nonemp1}
$\cS_-\neq \emptyset$ if and only if  $\cS_+\neq \emptyset$. More precisely, the following statements hold.
\begin{itemize}
\item [(1)] Let $v_-\in \cS_-$. Then the function $x\mapsto\lim_{t\ri\infty}T_t^+v_-(x)$ is well defined, and it belongs to $\cS_+$.
\item [(2)] Let $v_+\in \cS_+$. Then the function $x\mapsto\lim_{t\ri\infty}T_t^-v_+(x)$ is well defined, and it belongs to $\cS_-$.
\end{itemize}
\end{proposition}

In the following context of this section, we proceed under the assumption $\cS_-\neq \emptyset$. {Here we recall that the assumption $\cS_-\neq \emptyset$ is equivalent to assumption (A).}
It is well known that each $u_-\in \cS_-$ is semiconcave and $u_+\in \cS_+$ is semiconvex. For each $u_{\pm}\in \cS_{\pm}$, we define two subsets of $T^*M\times\R$ associated with $u_{\pm}$  respectively by
\begin{equation}\label{gv--}
G_{u_{\pm}}:=\mathrm{cl}\left(\big\{(x,p,u)\ |\ Du_{\pm}(x)\ \text{exists},\  u=u_{\pm}(x),\ p=Du_{\pm}(x)\big\}\right),
\end{equation}
where $\mathrm{cl}(A)$ denotes the closure of $A\subseteq T^*M\times \R$.
  %Similarly, for each $u_+\in \mathcal{S}_+$,  define a subset of $T^*M\times\R$ associated with $u_+$ by
%\begin{equation}\label{gv++}
%G_{u_+}:=\mathrm{cl}\left(\big\{(x,p,u)\ |\ Du_+(x)\ \text{exists},\  u=u_+(x),\ p=Du_+(x)\big\}\right).
%\end{equation}
Define
\[\tilde{\N}_{v_{\pm}}:=\tilde{\N}\cap G_{v_{\pm}},\quad \N_{v_{\pm}}:=\pi^*\tilde{\N}_{v_{\pm}}.
\]

The following   proposition shows a relation between  weak KAM solutions and semi-static curves. See \cite[Proposition 17]{WWY4} for details.

\begin{proposition}\label{pr119955}
	Let $(x(\cdot),u(\cdot)):\R\ri M\times\mathbb{R}$ be a semi-static curve.  Then there exists $v_-\in \cS_-$  (resp. $v_+\in \cS_+$) such that $u(t)=v_-(x(t))$ (resp. $u(t)=v_+(x(t)$)  for each $t\in \R$.
\end{proposition}

The following   proposition shows a relation between  weak KAM solutions and the Ma\~{n}\'{e} set. See  \cite[Theorem 1.3]{WWY4} for details.
\begin{proposition}\label{88996}
Let  $v_-\in \cS_-$, $v_+\in \cS_+$. Let \[\mathcal{I}_{v_-}:=\{x\in M\ |\ v_-(x)=\lim_{t\ri\infty}T_t^+v_-(x)\}\]
 \[\mathcal{I}_{v_+}:=\{x\in M\ |\ v_+(x)=\lim_{t\ri\infty}T_t^-v_+(x)\}.\]
 Then both $\mathcal{I}_{v_-}$ and $\mathcal{I}_{v_+}$ are not empty. Moreover,
\[\tilde{\N}_{v_{\pm}}=\{(x,p,u)\in T^*M\times\R\ |\ x\in \mathcal{I}_{v_{\pm}},\ u=v_{\pm}(x),\ p=Dv_{\pm}(x) \},\]
\[\tilde{\N}=\cup_{v_-\in \cS_-}\tilde{\N}_{v_-}=\cup_{v_+\in \cS_+}\tilde{\N}_{v_+}.\]
\end{proposition}

\vspace{1em}

\section{A technical lemma}\label{cll}
In this section, we are devoted to proving a technical lemma on the existence of certain transitive orbit. It will be used in the proofs of Theorem \ref{raomeome} and Theorem \ref{CONC}.
\begin{lemma}[Transitive criterion]\label{CONC1}
Given $X_1:=(x_1,p_1,u_1)\in \tilde{\A}$, $X_2:=(x_2,p_2,u_2)\in \tilde{\cS}_s$. If
\begin{equation}\tag{$\diamondsuit$}\label{himp1}
\lim_{t\to+\infty}h_{x_1,u_1}(x_2,t)=u_2,\ \limsup_{t\to +\infty}h^{x_2,u_2}(x_1,t)=u_1,
\end{equation}  then $X_1\rightsquigarrow X_2$.
\end{lemma}

%\begin{remark}
%By \cite[Theorem 1.4]{SWY}, for each $x_1,x_2\in M$, $u_1\in \R$, the limit $$\lim_{t\to+\infty}h_{x_1,u_1}(x_2,t)$$ does exist. It is only required that the limit is equal to $u_2$.
%\end{remark}

\begin{remark}
{ This criterion can be viewed as a technical trick. Loosely speaking, it means that under the condition ($\diamondsuit$), certain orbit $(x(\cdot),p(\cdot),u(\cdot)):\mathbb{R}\to T^*M\times\mathbb{R}$ can be controlled by its $x$ and $u$-arguments. Under (H1)-(H3) and (A), we consider the following two functions
 \[A(x):=\lim_{t\to+\infty}h_{x_1,u_1}(x,t),\quad B(x):=\limsup_{t\to +\infty}h^{x_2,u_2}(x,t).\]
 Due to (H3), the limit function $A:M\to \mathbb{R}$ is well defined. In general, $h^{x_2,u_2}(x,t)$ diverges as $t\to+\infty$.  But we have the existence of $B(x)$ if there exists a semi-static curve $(x(\cdot),u(\cdot)):\mathbb{R}\to M\times\mathbb{R}$ passing through $(x_2,u_2)$ (see Lemma \ref{uueequi}).
Then $A(x)$ and $B(x)$ can be viewed as two Peierls barriers for contact Hamiltonian systems.

If $H$ is independent of $u$, the condition ($\diamondsuit$) is reduced to
\[h^\infty(x_1,x_2)=u_2-u_1,\]
where $h^\infty:M\times M\to \mathbb{R}$ is the Peierls barrier in the classical case. Given any two points in the Aubry set in $T^*M$, this condition always holds if we choose a suitable embedding manner from $T^*M$ to $T^*M\times\mathbb{R}$.

In general cases, it is not easy to be verified. Under (H1)-(H3) and (A), we are able to use this criterion to prove the non-wandering property of the strongly static set, for which it suffices to consider the case with $X_1=X_2$. For the case with $X_1\neq X_2$, we consider this criterion in strictly increasing cases (H3'). In Theorem \ref{CONC}, it shows that the criterion can be verified by assuming ``if for each $v_+\in \mathcal{S}_+$, $v_+(x_2)=u_-(x_2)$ implies $v_+(x_1)=u_-(x_1)$". That shows the interplay between weak KAM solutions and the dynamics  around the Aubry set.}
\end{remark}

Let $\tilde{\mathcal{V}}$ be the set of $(x,p,u)\in T^*M\times\R$, for which there exists a strongly static orbit \[(x(\cdot),p(\cdot),u(\cdot)):\mathbb{R}\to T^*M\times\mathbb{R}\] passing through $(x,p,u)$.
Let ${\mathcal{V}}=\pi^* \tilde{\mathcal{V}}$.
By the definition of the strongly static set, \[\tilde{\mathcal{S}}_s=\mathrm{cl}(\tilde{\mathcal{V}}),\quad \mathcal{S}_s=\mathrm{cl}(\mathcal{V}).\]

To prove Lemma \ref{CONC1}, we only need to verify

\begin{lemma}\label{CONCLL1}
Given any $X_1:=(x_1,p_1,u_1)\in
\tilde{\A}$, $X_2:=(x_2,p_2,u_2)\in \tilde{\mathcal{V}}$. If
\begin{equation*}
\lim_{t\to+\infty}h_{x_1,u_1}(x_2,t)=u_2,\ \limsup_{t\to +\infty}h^{x_2,u_2}(x_1,t)=u_1,
\end{equation*}
then {$X_1\rightsquigarrow X_2$. }
\end{lemma}

We give a proof that Lemma \ref{CONCLL1} implies Lemma \ref{CONC1}.
%\begin{proof}
%Let $B(X,R)$ stand for the open metric ball on $T^*M\times\R$ centered at $X\in T^*M\times\R$ with radius $R$, and let $\bar B(X,R)$ stand for its closure.
%
%
%Given $X_1:=(x_1,p_1,u_1)$, $X_2:=(x_2,p_2,u_2)\in \tilde{\cS}_s$. For any neighborhoods $U_1$ of $X_1$ and  $U_2$ of $X_2$, one can find $R_1,R_2>0$ such that $\bar B(X_1,R_1)\subset U_1$, $\bar B(X_2,R_2)\subset U_2$. Note that $\tilde{\mathcal{S}}_s=\bar{\tilde{\mathcal{V}}}$. Thus, there exist $\{Y_n\}_{n\in \mathbb{N}}\subseteq \tilde{\mathcal{V}}$ and $\{Z_n\}_{n\in \mathbb{N}}\subseteq \tilde{\mathcal{V}}$ such that
%  \[Y_n\to X_1,\quad Z_n\to X_2,\quad n\to \infty.\]
%Hence,  there exists $N:=N(R_1,R_2)>0$ such that
%\[d(X_1,Y_N)\leq \frac{R_1}{4},\quad d(X_2,Z_N)\leq \frac{R_2}{4},\]
%which implies
%\[B(Y_N,\frac{R_1}{4})\subset B(X_1,R_1),\quad B(Z_N,\frac{R_2}{4})\subset B(X_2,R_2).\]
%By Definition \ref{ccnnet}, the transitive orbit between $Y_N$ and $Z_N$ also connects  $X_1$ and $X_2$.
%\end{proof}
\begin{proof}
Let $B(X,R)$ stand for the open  ball on $T^*M\times\R$ centered at $X$ with radius $R$, and let $\bar B(X,R)$ stand for its closure.

Given any $X_2:=(x_2,p_2,u_2)\in \tilde{\cS}_s$. For any neighborhood   $U$ of $X_2$, one can find $R>0$ such that $\bar B(X_2,R)\subset U$. Note that $\tilde{\mathcal{S}}_s=\mathrm{cl}(\tilde{\mathcal{V}})$. Thus, there exists a sequence $\{Z_n\}_{n\in \mathbb{N}}\subseteq \tilde{\mathcal{V}}$ such that
  \[ Z_n\to X_2,\quad n\to \infty.\]
Hence,  there exists $N:=N(R)>0$ such that
\[d(X_2,Z_N)\leq \frac{R}{4},\]
which implies
\[B\left(Z_N,\frac{R}{4}\right)\subset B(X_2,R).\]
By Definition \ref{ccnnet}, the existence of the transitive orbit from $X_1$ to $Z_N$ implies $X_1\rightsquigarrow X_2$.
\end{proof}

The following lemma gives a way to obtain one-sided semi-static curves from one-sided minimizing curves.
\begin{lemma}\label{mattt}
\
\begin{itemize}
\item [(1)]
Given $(x_0,u_0)\in M\times\R$, let $({x}(\cdot),{u}(\cdot)):\mathbb{R}_+\to M\times\mathbb{R}$ be a positively minimizing curve with $(x(0),u(0))=(x_0,u_0)$. If for each $t\geq 0$,
\begin{equation}
{u}(t)=\inf_{\tau>0}h_{x_0,u_0}({x}(t),\tau).
\end{equation}
Then for any $t_1$, $t_2\in\mathbb{R}_+$ with $t_1\leq t_2$, there holds
\begin{equation}\label{xxyymmtt}
	{u}(t_2)=\inf_{\tau>0}h_{{x}(t_1),{u}(t_1)}({x}(t_2),\tau).
	\end{equation}
\item [(2)]
Given $(x_0,u_0)\in M\times\R$, let $({x}(\cdot),{u}(\cdot)):\mathbb{R}_-\to M\times\mathbb{R}$ be a negatively minimizing curve with $(x(0),u(0))=(x_0,u_0)$. If for each $t\geq 0$,
\begin{equation}
{u}(-t)=\sup_{\tau>0}h^{x_0,u_0}({x}(-t),\tau).
\end{equation}
Then for any $t_1$, $t_2\in\mathbb{R}_+$ with $t_1\geq t_2$, there holds
\begin{equation}
	{u}(-t_1)=\sup_{\tau>0}h^{{x}(-t_2),{u}(-t_2)}({x}(-t_1),\tau).
	\end{equation}
\end{itemize}
\end{lemma}
\begin{proof}
We only prove Item (1). Item (2) follows from a similar argument. Since $({x}(\cdot),{u}(\cdot)):\mathbb{R}_+\to M\times\mathbb{R}$ is positively minimizing, then
\begin{equation}\label{9uu}
{u}(t_2)\geq\inf_{\tau>0}h_{{x}(t_1),{u}(t_1)}({x}(t_2),\tau), \quad \forall\ 0\leq t_1<t_2.
\end{equation}
By assumption, for each $t\geq 0$,
\begin{equation}
{u}(t)=\inf_{\tau>0}h_{x_0,u_0}({x}(t),\tau).
\end{equation}
It follows that
\[{u}(t_1)=\inf_{\tau>0}h_{x_0,u_0}({x}(t_1),\tau),\quad {u}(t_2)=\inf_{\tau>0}h_{x_0,u_0}({x}(t_2),\tau),\]
which gives rise to
\begin{align*}
{u}(t_2)&=\inf_{\tau>0}h_{x_0,u_0}({x}(t_2),\tau)\leq \inf_{\tau>0}h_{x_0,u_0}({x}(t_2),t_1+\tau)\\
&\leq \inf_{\tau>0}h_{{x}(t_1),h_{x_0,u_0}({x}(t_1),t_1)}({x}(t_2),\tau)\\
&=\inf_{\tau>0}h_{{x}(t_1),{u}(t_1)}({x}(t_2),\tau).
\end{align*}
Combining with (\ref{9uu}), we get (\ref{xxyymmtt}).
\end{proof}

The following proposition shows that for certain minimizing orbits $(x(\cdot),p(\cdot),u(\cdot)):\R\ri T^*M\times\R$, $p(t)$ is uniquely determined by $(x(t),u(t))$ for all $t\in\R$. { Let $\rho: \tilde{\mathcal{V}}\to M\times\R$. This result implies $\rho$ is injective, which gives rise to the  graph property of $\tilde{\mathcal{V}}$ on $M\times\R$.} Its proof is postponed in \ref{pui}.
\begin{proposition}\label{lem3.2x}
	If $(x,p_0,u)\in {\tilde{\mathcal{V}}}$, $(x,p_+,u)\in \tilde{\N}^+$ (resp. $(x,p_-,u)\in \tilde{\N}^-$),   then $p_0=p_+$ (resp. $p_0=p_-$).
\end{proposition}

Under the assumptions (H1)-(H3), by the definitions of $h_{x_0,u_0}(x,t)$ and $h^{x_0,u_0}(x,t)$, we have
\begin{proposition}\label{relation}Given $(x_0,x,t)\in M\times M\times (0,+\infty)$, $u,v\in \R$.
\begin{itemize}
\item [(1)] for all $u$, $v\in\mathbb{R}$ and all $(x,t)\in M\times (0,+\infty)$, $|h_{x_0,u}(x,t)-h_{x_0,v}(x,t)|\leq |u-v|$;
\item [(2)] if $u\geq v$, then $h^{x_0,u}(x,t)-h^{x_0,v}(x,t)\geq u-v$.
\end{itemize}
\end{proposition}

\noindent{\it Proof of Lemma \ref{CONCLL1}.}
Let $\{t_n\}_{n\in \mathbb{N}}$ be a sequence satisfying
\[\lim_{t_n\to+\infty} h^{x_2,u_2}(x_1,t_n)=u_1.\]
Let $\gm_n:[0,t_n]\ri M$ be a curve $\gm_n(0)=x_1$ and $\gm_n(t_n)=x_2$ such that
\begin{equation}\label{schu}
 h^{x_2,u_2}(x_1,t_n)=u_2-\int_0^{t_n}L\big(\gamma_n(\tau),\dot{\gamma}_n(\tau),h^{x_2,u_2}(\gamma_n(\tau),t_n-\tau)\big)d\tau.
 \end{equation}
  Let $u_n(t):=h^{x_2,u_2}(\gm_n(t),t_n-t)$. Then $u_n(t_n)=u_2$ and $u_n(0)=h^{x_2,u_2}(x_1,t_n)$. Let
 \[p_n(t):=\frac{\partial L}{\partial \dot{x}}(\gm_n(t),\dot{\gm}_n(t),u_n(t)).\]
 We claim that there exists $C>0$ such that
\[\|p_n(0)\|_{x_1}\leq C,\quad \|p_n(t_n)\|_{x_2}\leq C.\]
In fact, since $u_n(0)\ri u_1$ as $n\ri +\infty$, combining with the compactness of $M$, then \[u_n(1)=h^{x_2,u_2}(\gm_n(1),t_n-1)=h_{x_1,u_n(0)}(\gm_n(1),1)\] is bounded {independently on} $n$.  Similarly, $u_n(t_n-1)=h^{x_2,u_2}(\gm_n(t_n-1),1)$  are bounded independent of $n$. Note that $u_n(t_n)=u_2$. Then  one can find $C>0$ such that both $p_n(0)$ and $p_n(t_n)$ is bounded by $C$ (see \cite[Appendix, proof of Lemma 2.1]{WWY1} for details).

We assume, up to a subsequence, \[(\gm_n(0),p_n(0),u_n(0))\ri (x_1, p'_1,u_1),\quad (\gm_n(t_n),p_n(t_n),u_n(t_n))\ri (x_2,p'_2,u_2).\]
In order to prove Lemma \ref{CONCLL1}, it suffices to verify
\[p'_1=p_1,\quad p'_2=p_2.\]

First of all, we prove $p'_1=p_1$. By Proposition \ref{lem3.2x}, we only need to show $(x_1,p'_1,u_1)\in \tilde{\N}^+$. Let
\[(\bar{x}(t),\bar{p}(t),\bar{u}(t)):=\Phi_t(x_1,p'_1,u_1),\quad \forall t\geq 0.\]
By the definition of $\tilde{\N}^+$, we need to show that
 \begin{itemize}
 \item [(1)] the curve $(\bar{x}(\cdot),\bar{u}(\cdot)):\mathbb{R}_+\to M\times\mathbb{R}$ is positively minimizing;
  \item [(2)] for any $t_1$, $t_2\in\mathbb{R}_+$ with $t_1\leq t_2$, there holds
\begin{equation}\label{xxyymmtt11}
	\bar{u}(t_2)=\inf_{\tau>0}h_{\bar{x}(t_1),\bar{u}(t_1)}(\bar{x}(t_2),\tau).
	\end{equation}
\end{itemize}
Since $\gm_n:[0,t_n]\ri M$ satisfies (\ref{schu}), by the definition of $u_n(t)$, we have
\[u_n(t_1)=h^{\gm_n(t_2),u_n(t_2)}(\gm_n(t_1),t_2-t_1), \quad \forall 0\leq t_1<t_2\leq t_n.\]
By the continuous dependence of  solutions to ODEs on the initial data,
\[(\gm_n(t_1),u_n(t_1))\to (\bar{x}(t_1),\bar{u}(t_1)), \quad (\gm_n(t_2),u_n(t_2))\to (\bar{x}(t_2),\bar{u}(t_2)),\]
which combining with the  Lipschitz continuity of $(x_0,u_0,x)\mapsto h^{\cdot,\cdot}(\cdot,t_2-t_1)$ yields
\[\bar{u}(t_1)=h^{\bar{x}(t_2),\bar{u}(t_2)}(\bar{x}(t_1),t_2-t_1), \quad \forall 0\leq t_1<t_2<+\infty.\]
By Remark \ref{rcpps}, $(\bar{x}(\cdot),\bar{u}(\cdot)):\mathbb{R}_+\to M\times\mathbb{R}$ is positively minimizing. In order to verify Item (2), by Lemma \ref{mattt}(1), we need to prove
that for each $t\geq 0$,
\begin{equation}\label{cintroprov11}
\bar{u}(t)=\inf_{\tau>0}h_{x_1,u_1}(\bar{x}(t),\tau).
\end{equation}
By definition, we have
 \[\bar{u}(t)\geq\inf_{\tau>0}h_{x_1,u_1}(\bar{x}(t),\tau).\]
 It suffices to show
 \[\bar{u}(t)\leq\inf_{\tau>0}h_{x_1,u_1}(\bar{x}(t),\tau).\]
 By contradiction, we assume there exist $t_0,\tau_0>0$ such that
\[\bar{u}(t_0)>h_{x_1,u_1}(\bar{x}(t_0),\tau_0).\]
By Proposition \ref{relation}(2), one can find $\delta>0$ such that
\begin{equation}\label{sim}
h^{\bar{x}(t_0),\bar{u}(t_0)}(x_1,\tau_0)=u_1+\delta.
\end{equation}
For each $\eps>0$, there exists $n$ large enough such that
\[d((\gm_n(t_0),p_n(t_0),u_n(t_0)),(\bar{x}(t_0),\bar{p}(t_0),\bar{u}(t_0)))<\eps,\]
which  combining with Lipschitz continuity of $(x_0,u_0)\mapsto h^{\cdot,\cdot}(x_1,\tau_0)$  implies
\[h^{\gm_n(t_0),u_n(t_0)}(x_1,\tau_0)\geq u_1+\frac{\delta}{2}.\]
Note that $\gm_n(t_n)=x_2$, $\gm_n(0)=x_1$, $u_n(t_n)=u_2$. It follows that
\begin{align*}
h^{x_2,u_2}(x_1,t_n-t_0+\tau_0)&=h^{\gm_n(t_n),u_n(t_n)}(\gm_n(0),t_n-t_0+\tau_0)\\
&\geq h^{\gm_n(t_0),h^{\gm_n(t_n),u_n(t_n)}(\gm_n(t_0),t_n-t_0)}(\gm_n(0),\tau_0)\\
&=h^{\gm_n(t_0),u_n(t_0)}(\gm_n(0),\tau_0)\\
&\geq u_1+\frac{\delta}{2}.
\end{align*}
By assumption,  $\limsup_{t\to+\infty}h^{x_2,u_2}(x_1,t)= u_1$.
Letting $n\ri+\infty$, we have $u_1>u_1$, which is a contradiction.

Next, we prove $p'_2=p_2$. By Proposition \ref{lem3.2x}, we only need to show $(x_2,p'_2,u_2)\in \tilde{\N}^-$. Let
\[(\tilde{x}(-t),\tilde{p}(-t),\tilde{u}(-t)):=\Phi_{-t}(x_2,p'_2,u_2),\quad \forall 0\leq t<+\infty.\]
We aim to show that
 \begin{itemize}
 \item [(1)] the curve $(\tilde{x}(\cdot),\tilde{u}(\cdot)):\mathbb{R}_-\to M\times\mathbb{R}$ is negatively minimizing;
  \item [(2)] for any $t_1$, $t_2\in\mathbb{R}_+$ with $t_1\geq t_2$, there holds
\begin{equation}\label{xxyymmttwang}
	\tilde{u}(-t_1)=\sup_{\tau>0}h^{\tilde{x}(-t_2),\tilde{u}(-t_2)}(\tilde{x}(-t_1),\tau).
	\end{equation}
\end{itemize}
The proof of Item (1) is similar to the one of $p'_1=p_1$ above.

Let $\eta_n(-t):=\gm_n(-t+t_n)$ for each $t\geq 0$. It follows that  $\eta_n:[-t_n,0]\ri M$ satisfies (\ref{schu}) with $\eta_n(0)=x_2$, $\eta_n(-t_n)=x_1$. Let $v_n(-t):=h^{x_2,u_2}(\eta_n(-t),t)$ for each $t\geq 0$. Then $v_n(0)=u_2$ and
\begin{equation}\label{vn-t}
v_n(-t_1)=h^{\eta_n(-t_2),v_n(-t_2)}(\eta_n(-t_1),t_1-t_2), \quad \forall 0\leq t_2<t_1<t_n.
\end{equation}
%which combining with Lipschitz continuity of $(x_0,u_0,x)\mapsto h^{\cdot,\cdot}(\cdot,t_1-t_2)$ yields
%\[\tilde{u}(-t_1)=h^{\tilde{x}(-t_2),\tilde{u}(-t_2)}(\tilde{x}(-t_1),t_1-t_2), \quad \forall 0\leq t_2<t_1<+\infty.\]
%By Remark \ref{rcpps}, $(\tilde{x}(\cdot),\tilde{u}(\cdot)):\mathbb{R}_-\to M\times\mathbb{R}$ is negatively minimizing. In order to show  $p'_2=p_2$, by Lemma \ref{mattt}(2), it suffices to verify that
% for each $t\geq 0$,
%\begin{equation}\label{cintroprovwang}
%\tilde{u}(-t)=\sup_{\tau>0}h^{x_2,u_2}(\tilde{x}(-t),\tau).
%\end{equation}
%By definition,
% \[\tilde{u}(-t)\leq\sup_{\tau>0}h^{x_2,u_2}(\tilde{x}(-t),\tau).\]
% It suffices to show
% \[\tilde{u}(-t)\geq\sup_{\tau>0}h^{x_2,u_2}(\tilde{x}(-t),\tau).\]
% By contradiction, we assume there exist $t_0,\tau_0>0$ such that
%\[\tilde{u}(-t_0)<h^{x_2,u_2}(\tilde{x}(-t_0),\tau_0).\]
%By Proposition \ref{relation}(1), one can find $\delta>0$ such that
%\[h_{\tilde{x}(-t_0),\tilde{u}(-t_0)}(x_2,\tau_0)=u_2-\delta.\]
%For each $\eps>0$, there exists $n$ large enough such that
%\[d((\eta_n(-t_0),p_n(-t_0),v_n(-t_0)),(\tilde{x}(-t_0),\tilde{p}(-t_0),\tilde{u}(t_0)))<\eps,\]
%which  combining with Lipschitz continuity of $(x,u)\mapsto h^{\cdot,\cdot}(x_2,\tau_0)$  implies
%\begin{equation}\label{xxyrr-t}
%h_{\eta_n(-t_0),v_n(-t_0)}(x_2,\tau_0)\leq u_2-\frac{\delta}{2}.
%	\end{equation}
By (\ref{vn-t}),
\begin{equation}\label{xxvn-t}
v_n(-t_0)=h_{\eta_n(-t_n),v_n(-t_n)}(\eta_n(-t_0),t_n-t_0).
	\end{equation}
Similar to the proof of (\ref{sim}), we have
\begin{equation}\label{xxyrr-t}
h_{\eta_n(-t_0),v_n(-t_0)}(x_2,\tau_0)\leq u_2-\frac{\delta}{2}.
	\end{equation}
By Proposition \ref{relation}(1),
\begin{equation}\label{566}
\begin{split}
&|h_{\eta_n(-t_0),h_{x_1,u_1}(\eta_n(-t_0),t_n-t_0)}(x_2,\tau_0)- h_{\eta_n(-t_0),h_{x_1,v_n(-t_n)}(\eta_n(-t_0),t_n-t_0)}(x_2,\tau_0)|\\
\leq &|h_{x_1,u_1}(\eta_n(-t_0),t_n-t_0)-h_{x_1,v_n(-t_n)}(\eta_n(-t_0),t_n-t_0)|\\
\leq &|u_1-v_n(-t_n))|.
\end{split}
\end{equation}
Then we have
\begin{align*}
h_{x_1,u_1}(x_2,t_n-t_0+\tau_0)&\leq h_{\eta_n(-t_0),h_{x_1,u_1}(\eta_n(-t_0),t_n-t_0)}(x_2,\tau_0)\\
&\leq h_{\eta_n(-t_0),h_{x_1,v_n(-t_n)}(\eta_n(-t_0),t_n-t_0)}(x_2,\tau_0)+|u_1-v_n(-t_n)|\\
&=h_{\eta_n(-t_0),v_n(-t_0)}(x_2,\tau_0)+|u_1-v_n(-t_n)|\\
&\leq u_2-\frac{\delta}{2}+|u_1-v_n(-t_n)|,
\end{align*}
where the first inequality is from the Markov property of $h_{x_1,u_1}(x,t)$, the second  inequality is owing to (\ref{566}), the first equality is from (\ref{xxvn-t}),  and the last inequality is from (\ref{xxyrr-t}).

Note that $\eta_n(-t_n)=x_1$, $\eta_n(0)=x_2$, $v_n(0)=u_2$.
By assumption, $h^{x_2,u_2}(x_1,t_n)\ri u_1$ as $t_n\ri +\infty$, then $v_n(-t_n)\rightarrow u_1$. It is also assumed that $\lim_{t\to+\infty}h_{x_1,u_1}(x_2,t)=u_2$.
Letting $n\ri+\infty$, we have $u_2<u_2$, which is a contradiction.

This completes the proof of Lemma \ref{CONCLL1}.\EEnd

%\begin{remark}
%Based on the proof of Lemma \ref{CONCLL1}, the assumption (H3) is only needed for (\ref{566}).
%\end{remark}
\vspace{1ex}

\section{Topological dynamics around the Aubry set}\label{se44}
In this section, we are devoted to proving Theorem \ref{raomeome}.
\subsection{Mather set and recurrence}\label{mathre}
 For each $v\in \cS_-$, $\tilde{\mathcal{N}}_{v}$ is a flow invariant subset of $T^*M\times\mathbb{R}$. By Proposition \ref{88996},
  the set $\tilde{\mathcal{N}}_{v}$ is non-empty and compact. Then
 $\tilde{\mathcal{M}}\neq\emptyset$ directly follows from the definition of the Mather set and the assumption (A). Next, we prove
\[\tilde{\mathcal{M}}\subseteq \tilde{\mathcal{N}}\cap \mathrm{cl}(\tilde{\mathcal{R}}).\]
 Note that Mather measures are  invariant Borel probabilities. Let $\mu$ be a Mather measure. By the Poincar\'{e} recurrence theorem,  one can find a set $A\subseteq \tilde{\N}$ of total $\mu$-measure such that if $(x_0,p_0,u_0)\in A$, then there exist $\{t_m\}_{m\in \mathbb{N}}$ and $\{t_n\}_{n\in \mathbb{N}}$   such that
\[d\left((x_0,p_0,u_0),\Phi_{t_m}(x_0,p_0,u_0)\right)\ri 0\quad \text{as}\ \ t_m\ri +\infty,\]
\[d\left((x_0,p_0,u_0),\Phi_{t_n}(x_0,p_0,u_0)\right)\ri 0\quad \text{as}\ \ t_n\ri -\infty.\]
%Let
%\[(x(t),p(t),u(t)):=\Phi_t(x_0,p_0,u_0),\quad \forall t\in \R.\]
%By definition, $(x(\cdot),u(\cdot)):\mathbb{R}\to M\times\mathbb{R}$ is a semi-static curve.
Since $\tilde{\N}$ is closed and $A$ is dense in supp$(\mu)$, then we have
\[\tilde{\mathcal{M}}\subseteq \tilde{\mathcal{N}}\cap \mathrm{cl}(\tilde{\mathcal{R}}).\]

\subsection{Recurrence and strong staticity}\label{recstr}
In this part, we aim to show
\[\tilde{\N}\cap \mathrm{cl}(\tilde{\mathcal{R}})\subseteq \tilde{\mathcal{S}}_s.\]
Let $(x(\cdot),u(\cdot)):\mathbb{R}\to M\times\mathbb{R}$ be a semi-static curve. By definition, for each $t_1\leq t_2$,
\begin{equation}\label{u89890}
	u(t_2)=\inf_{s>0}h_{x(t_1),u(t_1)}(x(t_2),s).
\end{equation}
By Proposition \ref{pr11}, we have
\begin{equation}\label{uhxtt6655}
	u(t_1)=\sup_{s>0}h^{x(t_2),u(t_2)}(x(t_1),s).
\end{equation}
It remains to prove that for each $t_1>t_2$, (\ref{uhxtt6655}) still holds.

By Proposition \ref{pr119955}, there exists $v_-\in \cS_-$ such that $u(t)=v_-(x(t))$ for all $t\in\R$.  If $x(t_1)=x(t_2)$, then \[u(t_1)=v_-(x(t_1))=v_-(x(t_2))=u(t_2),\] for which (\ref{uhxtt6655}) holds for each $t_1>t_2$.

In the following, we prove the case with $x(t_1)\neq x(t_2)$. Note that $t_1>t_2$. Since $(x(\cdot),u(\cdot)):\mathbb{R}\to M\times\mathbb{R}$ is semi-static, we have
	\begin{equation}\label{uhxtt775566}
	\quad u(t_2)=\sup_{s>0}h^{x(t_1),u(t_1)}(x(t_2),s).
\end{equation}
 Let $p(t):=\frac{\partial L}{\partial \dot{x}}(x(t),\dot{x}(t),u(t))$. We assume
\[(x(t_1),p(t_1),u(t_1))\in \tilde{\N}\cap \tilde{\mathcal{R}},\]
Let $\{\tau_n\}_{n\in \mathbb{N}}$ be a sequence such that
\[d\left((x(t_1),p(t_1),u(t_1)),\Phi_{\tau_n}(x(t_1),p(t_1),u(t_1))\right)\ri 0\quad \text{as}\ \tau_n\ri -\infty.\]
 Denote $\Delta:=t_2-t_1$. By (\ref{uhxtt775566}), if  $\tau_n<\Delta$, we have
\[u(t_1+\tau_n)=\sup_{s>0}h^{x(t_2),u(t_2)}(x(t_1+\tau_n),s).\]
Note that  $x(t_1)\neq x(t_2)$. It follows from the Lipschitz continuity of Ma\~{n}\'{e} potentials (see Proposition \ref{immm} below) that
\[u(t_1)=\sup_{s>0}h^{x(t_2),u(t_2)}(x(t_1),s),\]
which together with (\ref{uhxtt775566}) implies
 $(x(\cdot),u(\cdot)):\mathbb{R}\to M\times\mathbb{R}$ is a strongly static curve.
Then $\tilde{\N}\cap \mathrm{cl}(\tilde{\mathcal{R}})\subseteq \tilde{\mathcal{S}}_s$ follows from the closedness of $\tilde{\mathcal{S}}_s$.

\subsection{Strong staticity and non-wandering property}\label{strsnon}

In this part, we are devoted to proving
\[\tilde{\mathcal{S}}_s\subseteq \tilde{\A}\cap \tilde{\Omega}.\]
 Let $(x(\cdot),u(\cdot)):\mathbb{R}\to M\times\mathbb{R}$ be a strongly static curve. Let $p(0):=\frac{\partial L}{\partial \dot{x}}(x(0),\dot{x}(0),u(0))$. Let $Q_0:=(x(0),p(0),u(0))$. By definition, it suffices to prove that for each neighborhood $U_n$ of $Q_0$, there exist $Q_n\in U_n$ and $T_n>0$ such that $\Phi_{T_n}(Q_n)\in U_n$.

By \cite[Theorem 1.4]{SWY}, under (H1)-(H3) and (A),
%\begin{proposition}\label{con}
%	For each $\varphi\in C(M,\mathbb{R})$, the uniform limit $\lim_{t\rightarrow +\infty}T^-_t\varphi(x)$ exists. Let \[u_\infty(x):=\lim_{t\rightarrow +\infty}T^-_t\varphi(x)\] for each  $x\in M$. Then $u_\infty(x)$ is a viscosity solution of equation (\ref{hj}).
%\end{proposition}
% Thus,
 the following function
\begin{align}\label{ba}
h_{x_0,u_0}(x,+\infty):=\lim_{t\rightarrow +\infty}h_{x_0,u_0}(x,t),\quad x\in M
\end{align}
is well defined. Moreover, for each $s,t\in \R$,  both $\lim_{\tau\to+\infty}h_{x(s),u(s)}(x(t),\tau)$. Unfortunately, $\lim_{t\rightarrow +\infty}h^{x_0,u_0}(x,t)$ is not always well defined. But it can be proved that  $\limsup_{\tau\ri+\infty}h^{x(s),u(s)}(x(t),\tau)$ is well defined (see Lemma \ref{uueequi} below).

\begin{lemma}\label{lem3.1qq}
Let $(x(\cdot),u(\cdot)):\R\ri M\times\mathbb{R}$ be a semi-static curve, then
\begin{itemize}
\item [(1)] it is static if and only if
\[u(t)=\lim_{\tau\to+\infty}h_{x(s),u(s)}(x(t),\tau),\quad \forall s,\ t\in\mathbb{R};\]
\item [(2)] it is strongly static if and only if
	\[
	u(t)=\limsup_{\tau\ri+\infty}h^{x(s),u(s)}(x(t),\tau),\quad \forall s,\ t\in\mathbb{R}.
	\]
\end{itemize}
\end{lemma}

\begin{proof} We only prove Item (1). Item (2) follows a similar argument.
By  definition, we have
\[u(t)=\inf_{\tau>0}h_{x(s),u(s)}(x(t),\tau)\leq \lim_{\tau\to+\infty}h_{x(s),u(s)}(x(t),\tau),\quad \forall s,\ t\in\mathbb{R}.\]

	On the other hand, for each $n\in\mathbb{N}$, we get
	\[
	u(t)=\inf_{\sigma>0}h_{x(s+n),u(s+n)}(x(t),\sigma).
	\]
	There is a sequence $\{\sigma_n\}\subset\mathbb{R}_+$ such that
	\[
	h_{x(s+n),u(s+n)}(x(t),\sigma_n)<u(t)+\frac{1}{n},
	\]
	which together with the Markov  property implies
	\begin{align*}
	h_{x(s),u(s)}(x(t),n+\sigma_n)&\leq h_{x(s+n),h_{x(s),u(s)}(x(s+n),n)}(x(t),\sigma_n)\\
&=h_{x(s+n),u(s+n)}(x(t),\sigma_n)\\
&<u(t)+\frac{1}{n}.
	\end{align*}
	Let $n\to+\infty$. Then
	\[
	u(t) \geq \lim_{\tau\ri+\infty}h_{x(s),u(s)}(x(t),\tau).
	\]
Then we have
\[
	u(t)=\lim_{\tau\ri+\infty}h_{x(s),u(s)}(x(t),\tau),\quad \forall s,\ t\in\mathbb{R}.
	\]

Conversely, if \[u(t)=\lim_{\tau\to+\infty}h_{x(s),u(s)}(x(t),\tau),\quad \forall s,\ t\in\mathbb{R},\]
then
\[u(t)\geq \inf_{\tau>0}h_{x(s),u(s)}(x(t),\tau).\]
Note that $(x(\cdot),u(\cdot)):\R\ri M\times\mathbb{R}$ is  semi-static. By Proposition \ref{pr119955}, there exists $v_-\in \cS_-$   such that $u(t)=v_-(x(t))$ for each $t\in \R$. Combining Proposition \ref{pr4.2} and Proposition \ref{pr4.5}, we have
\begin{align*}
u(t)&=v_-(x(t))=T_\tau^-v_-(x(t))=\inf_{y\in M}h_{y,v_-(y)}(x(t),\tau)\\
&\leq h_{x(s),v_-(x(s))}(x(t),\tau)=h_{x(s),u(s)}(x(t),\tau),\quad \forall \tau>0,
\end{align*}
which implies
\[u(t)\leq \inf_{\tau>0}h_{x(s),u(s)}(x(t),\tau).\]
This completes the proof of Lemma \ref{lem3.1qq}.
\end{proof}

Let $x_0:=x(0)$, $u_0:=u(0)$. By Lemma \ref{lem3.1qq},  \[\lim_{t\to+\infty}h_{x_0,u_0}(x_0,t)=u_0,\ \limsup_{t\to +\infty}h^{x_0,u_0}(x_0,t)=u_0.\]
%By definition,
% \begin{equation}\label{sisisi33}
%h^{x_0,u_0}(x_0,t_n)\leq \sup_{\tau>0}h^{x_0,u_0}(x_0,\tau)=u_0.
%	\end{equation}
Then  $\tilde{\mathcal{S}}_s\subseteq \tilde{\A}\cap \tilde{\Omega}$ follows from Lemma \ref{CONC1}.

\begin{remark}\label{remk4}
The Aubry set  in the classical case is defined in $T^*M$ instead of $T^*M\times\R$. From weak KAM point of view (\cite[Theorem 5.2.8]{Fat-b}), there exists a conjugate pair $(u_-,u_+)$ such that the Aubry set is represented as
\[\tilde{\mathcal{I}}_{(u_-,u_+)}:=\{(x,p)\in T^*M\ |\ u_-(x)=u_+(x),\ p=Du_-(x)\}.\]
 One can embed this set from $T^*M$ into $T^*M\times\R$ by adding the $u$-argument to get the {embedded} Aubry set denoted by $\tilde{\mathcal{A}}_e$ in the following way.
\[\tilde{\mathcal{A}}_e:=\{(x,p,u)\in T^*M\times\R\ |\ (x,p)\in \tilde{\mathcal{I}}_{(u_-,u_+)},\ u=u_-(x).\}\]
Note that $\tilde{\mathcal{A}}_e\subset\tilde{\mathcal{A}}=\tilde{\cS}_s$ in the classical case. From Theorem \ref{raomeome}, the {embedded} Aubry set is non-wandering  in $T^*M\times\R$ in classical cases. Moreover, the non-wandering set in $T^*M\times\R$ is also an {embedded} Aubry set by choosing certain conjugate pair (see \cite[Theorem 1.5]{L}). It means that the embedded Aubry set on $T^*M\times\R$ is the same as the non-wandering set generated by
\begin{align}\label{1x}
\left\{
        \begin{array}{l}
        \dot{x}=\frac{\partial H}{\partial p}(x,p),\\
        \dot{p}=-\frac{\partial H}{\partial x}(x,p),\\
        \dot{u}=\frac{\partial H}{\partial p}(x,p)\cdot p-H(x,p).
         \end{array}
         \right.
\end{align}
 Let $\Pi^*:T^*M\times\R\to T^*M$ be the standard projection. By definition, the projection of the non-wandering set  generated by (\ref{1x})  is exactly the Aubry set in the classical case. From this point of view, we gives a description for the Aubry set in the classical case without using action minimizing property.
%In general contact cases, a non-wandering point may not be action minimizing, for example,  the focal point in the phase space of the dissipative pendulum.
%In fact, the properties of the embedding Aubry set can be obtained in a more direct way. If $H$ is independent of $u$,
% then (see \cite[Remark 4]{WWY4})
%\[h_{x_1,u_1}(x_2,t)=u_1+h^t(x_1,x_2),\ h^{x_2,u_2}(x_1,t)=u_2-h^t(x_1,x_2).\]
%The condition (\ref{himp1}) is reduced to
%\begin{equation}\label{himp133}
%\liminf_{t\to +\infty}h^t(x_1,x_2)=u_2-u_1.
%\end{equation}
%From classical weak KAM theory,  there exists a conjugate pair $(u^*_-,u^*_+)$,
%\[h^{\infty}(x,y)=u^*_-(y)-u^*_+(x),\ \forall\ x,y\in \A,\]
%where $\A$ denotes the projected Aubry set, and $h^{\infty}(x,y)$ denotes the Peierls barrier. By the definition of $h^{\infty}(x,y)$, (\ref{himp133}) holds for any  $X_1, X_2\in \tilde{\mathcal{A}}$. By Lemma \ref{CONC1}, there exists a transitive orbit between any two points in the embedding Aubry set. In particular, it is a non-wandering set.
\end{remark}

\vspace{1em}

\section{Strictly increasing case}\label{se55}
In this section, we consider the cases under (H1), (H2), (H3') and (A).
\subsection{The structure of $\cS_+$}
{We give a proof} of Theorem \ref{pooo} in this part.
It was shown by \cite[Proposition 12]{WWY3} that
\begin{proposition}\label{W1}
All of elements in $\mathcal S_+$ are uniformly bounded and equi-Lipschitz continuous.
\end{proposition}
%We denote the Lipschitz constant by $\kappa$.
 Note that $(\mathcal{S}_+,\preceq)$ is a partially ordered set. In view of  Zorn's lemma, if every chain in $\mathcal{S}_+$ has a lower bound in $\mathcal{S}_+$, then $\mathcal{S}_+$ contains a minimal element. To prove Item (1) of Theorem \ref{pooo}, it is suffices to show

\begin{proposition}\label{minS+}
Let $\mathcal{Z}$ be a totally ordered subset of $\mathcal{S}_+$. Let $\check{u}(x):=\inf_{u\in \mathcal{Z}}u(x)$ for each $x\in M$. Then $\check{u}\in \mathcal{S}_+$.
\end{proposition}

\begin{lemma}\label{micons}
There exists a sequence $\{u_n\}_{n\in\mathbb N}\subset \mathcal{Z}$ such that $u_n$ converges to $\check{u}$ uniformly.
\end{lemma}
\begin{proof}
Note that all of elements in $\mathcal S_+$ are uniformly bounded and $\kappa$-equi-Lipschitz continuous. We only need to construct a sequence $\{u_n\}_{n\in\mathbb N}\subset \mathcal{Z}$ such that $u_n$ converges to $\check{u}$ pointwisely.

Since the Riemannian manifold $M$ is compact, it is  separable. Namely on can find a countable dense subset denoted by $U:=\{x_1,x_2,\dots,x_n,\dots\}$.

\bigskip
\noindent \textbf{Claim.} There exists a sequence $\{u_n\}_{n\in\mathbb N}\subset \mathcal{Z}$ such that for a given $n\in\mathbb N$ and each $i\in\{1,2,\dots,n\}$,
\begin{equation}\label{un}
  0\leq u_n(x_i)-\check u(x_i)<\frac{1}{n}.
\end{equation}

\noindent If  the claim is true, then  $u_n$ converges to $\check{u}$ pointwisely. In fact, according to Proposition \ref{W1}, every $u_+\in\mathcal S_+$ is $\kappa$-Lipschitz, we have
\begin{equation}\label{baru}
  \check u(x)-\check u(y)\leq \sup_{u_+\in\mathcal S_+}|u_+(x)-u_+(y)|\leq \kappa \text{dist}(x,y).
\end{equation}
Fix $x\in M$. There exists a subsequence $V:=\{x_{k_m}\}_{m\in\mathbb N}\subseteq U$ such that $|x_{k_m}-x|<1/{k_m}$. Given ${k}_m\in \mathbb{N}$, we take $n\geq k_m$. Then $\{x_1,x_2,\dots,x_n\}\cap V\neq \emptyset$. Let ${x}_{i_0}\in \{x_1,x_2,\dots,x_n\}\cap V$. It follows from (\ref{baru}) that
\begin{equation*}
\begin{aligned}
  |u_n(x)-\check u(x)|&\leq |u_n(x)-u_n(x_{i_0})|+|u_n(x_{i_0})-\check u(x_{i_0})|+|\check u(x)-\check u(x_{i_0})|
  \\ &\leq 2\kappa\ \text{dist}(x_{i_0},x)+\frac{1}{n}\leq \frac{2\kappa}{{i_0}}+\frac{1}{n}.
\end{aligned}
\end{equation*}
Let $n$ and ${i_0}$ tend to $+\infty$ successively. We get the pointwise convergence of $u_n$ to $\check u$.

We prove the claim by induction in the following. For $x_1\in U$, by the definition of the infimum, there exists $u_1\in \mathcal{Z}$ such that $u_1(x_1)-\check u(x_1)<1/n$ for a given $n\in \mathbb{N}$. We assume there exists $u_k\in \mathcal{Z}$ such that for $i\in \{1,2,\ldots,k\}$, $u_k(x_i)-\check u(x_i)<1/n$. One needs to construct $u_{k+1}$. For $x_{k+1}\in U$, if $u_k(x_{k+1})-\check u(x_{k+1})<1/n$, then we take $u_{k+1}\equiv u_k$. Otherwise, we have $u_k(x_{k+1})-\check u(x_{k+1})\geq 1/n$. In this case, one can find $u_{k+1}\in \mathcal{Z}$ such that $u_{k+1}(x_{k+1})-\check u(x_{k+1})<1/n$.

It remains to show $u_{k+1}(x_{i})-\check u(x_{i})<1/n$ for  $i\in \{1,2,\ldots,k\}$. We know that
\[u_{k+1}(x_{k+1})<\check u(x_{k+1})+\frac{1}{n}\leq u_k(x_{k+1}).\]
Note that $\mathcal{Z}$ is totally ordered. It yields $u_{k+1}\leq u_k$ on $M$. Thus, for  $i\in \{1,2,\ldots,k\}$,
\[u_{k+1}(x_i)-\check u(x_{i})\leq u_{k}(x_i)-\check u(x_{i})<\frac{1}{n}.\]

This completes the proof of Lemma \ref{micons}.
\end{proof}

Under (H1)-(H3), by the definitions of $T_t^{\pm}$, we have the following properties (see \cite[Proposition 2.4]{WWY2} for details).
\begin{proposition}\label{psg}
\
\begin{itemize}
\item [(1)] For $\varphi_1$ and $\varphi_2 \in C(M)$, if $\varphi_1(x)<\varphi_2(x)$ for all $x\in M$, we have $T^-_t\varphi_1(x)< T^-_t\varphi_2(x)$ and $T^+_t\varphi_1(x)< T^+_t\varphi_2(x)$ for all $(x,t)\in M\times(0,+\infty)$.

\item [(2)] Given any $\varphi$ and $\psi\in C(M)$, we have $\|T^-_t\varphi-T^-_t\psi\|_\infty\leq\|\varphi-\psi\|_\infty$ and $\|T^+_t\varphi-T^+_t\psi\|_\infty\leq e^{\lambda t}\|\varphi-\psi\|_\infty$ for all $t>0$.
\end{itemize}
\end{proposition}

\noindent{\it Proof of Proposition \ref{minS+}.}
If $\mathcal{Z}$ is a finite set, the proof is finished. We then consider $\mathcal{Z}$ being a infinite set. By Proposition \ref{pr4.5}, we need to show $\check u$ is a fixed point of $T_t^+$. By Proposition \ref{psg} (2), we have
\begin{equation*}
  \|T^+_tu_n-T^+_t\check u\|_\infty\leq e^{\lambda t}\|u_n-\check u\|_\infty,
\end{equation*}
By Lemma \ref{micons}, the right hand side tends to zero. Then for a given $t>0$,
\begin{equation*}
  T^+_t\check u=\lim_{n\rightarrow+\infty}T^+_tu_n=\lim_{n\rightarrow+\infty}u_n=\check u,
\end{equation*}
which implies $\check u\in \cS_+$ by Proposition \ref{pr4.5}.
\EEnd

Next, we prove Item (2) of Theorem \ref{pooo}. Let us recall $\mathcal{Z}_{\max}$ denotes a maximal totally ordered subset  of $\cS_+$,  $u_-$ denotes the unique backward weak KAM solution, and $u^*$ denotes the minimal element in $\mathcal{Z}_{\max}$, i.e.
\[u^*(x):=\inf_{v_+\in \mathcal{Z}_{\max}}v_+(x).\]
\begin{lemma}\label{l77}
Denote
\[ {\mathcal{I}}_{u^*}:=\{x\in M\ |\ u^*(x)=u_-(x)\}.\]
Then  ${\mathcal{I}}_{u^*}\neq \emptyset$ and for all $v_+\in \mathcal{Z}_{\max}$, $v_+=u_-$ on ${\mathcal{I}}_{u^*}$.
\end{lemma}

\begin{proof}
By Proposition \ref{minS+}, $u^*\in \cS_+$. By Proposition \ref{nonemp1} and the uniqueness of $u_-$,
\[u_-(x)=\lim_{t\to+\infty}T_t^-u^*(x).\]
According to Proposition \ref{88996},
 ${\mathcal{I}}_{u^*}\neq \emptyset$.
Note that for all $v_+\in \mathcal{Z}_{\max}$,
\[u^*\leq v_+\leq u_-,\quad \text{on}\ M.\]
Then $v_+(x)=u_-(x)$ for all $x\in {\mathcal{I}}_{u^*}$.
\end{proof}

%By \cite[Lemma 6.3]{WWY4}, we have
\begin{lemma}\label{hyxuxttttt}
Define $m(\cdot,\cdot):M\times M\to \R$ by
\[m(x,y):=\limsup_{t\ri +\infty}h^{y,u_-(y)}(x,t),\quad \forall x,y\in M.\]
Then for each $y\in M$, $m(\cdot,y)\in \cS_+$.
\end{lemma}
\begin{proof}
Fix $y\in M$. By \cite[Lemma 1(i)]{WWY3}, $h^{y,u_-(y)}(\cdot,\cdot)$ is uniformly bounded on $M\times [\delta,+\infty)$ for any $\delta>0$. Thus, there exists a constant $K>0$ independent of $t$ such that for $t>\delta$ and each $x\in M$,
\[|h^{y,u_-(y)}(x,t)|\leq K.\]
Note that for any $t> 2\delta$, we have
\begin{align*}
&\left|h^{y,u_-(y)}(x,t)-h^{y,u_-(y)}(x',t)\right|\\
=&\left|\sup_{z\in M}h^{z,h^{y,u_-(y)}(z,t-\delta)}(x,\delta)-\sup_{z\in M}h^{z,h^{y,u_-(y)}(z,t-\delta)}(x',\delta)\right|\\
\leq &\sup_{z\in M}\left|h^{z,h^{y,u_-(y)}(z,t-\delta)}(x,\delta)-h^{z,h^{y,u_-(y)}(z,t-\delta)}(x',\delta)\right|.
\end{align*}
Since $h^{\cdot,\cdot}(\cdot,{\delta})$ is uniformly Lipschitz on $M\times [-K,K]\times M$ with some Lipschitz constant denoted by $\iota$, then
\[
\left|h^{y,u_-(y)}(x,t)-h^{y,u_-(y)}(x',t)\right|\leq \iota\ \text{dist}(x,x'), \quad \forall t> 2\delta.
\]
It follows that the family $\{h^{y,u_-(y)}(x,t)\}_{t>2\delta}$ is equi-Lipschitz continuous with respect to $x$. Thus, $m(x,y)$ is well defined.
 Note that for a given $t>0$, the  Lax-Oleinik semigroup $T_t^+$ satisfies
 \[\|T_t^+\varphi-T_t^+\psi\|_\infty\leq e^{\lambda t}\|\varphi-\psi\|_\infty,\]
 for any $\varphi,\psi\in C(M,\R)$. Note that $T_t^+$ commutes with $\limsup$. It follows that for a given $t\geq 0$,
 \[T_t^+m(x,y)=\limsup_{s\ri +\infty}T_t^+h^{y,u_-(y)}(x,s)=\limsup_{s\ri +\infty}h^{y,u_-(y)}(x,s+t)=m(x,y),\]
 which implies  $m(\cdot,y)\in \cS_+$.
\end{proof}
Let us recall ${\mathcal{V}}=\pi^* \tilde{\mathcal{V}}$, where  $\tilde{\mathcal{V}}$ denotes the set of $(x,p,u)\in T^*M\times\R$, for which there exists a strongly static orbit \[(x(\cdot),p(\cdot),u(\cdot)):\mathbb{R}\to T^*M\times\mathbb{R}\] passing through $(x,p,u)$.

\begin{lemma}\label{reminx}Given $y\in \mathcal{V}$, we have
\begin{equation}\label{wl5}
\inf_{\substack{v_+(y)=u_-(y)\\  v_+\in \cS_+ } }v_+(x)=m(x,y),\quad \forall x\in M.
\end{equation}
\end{lemma}

\begin{proof}
By the definition of $\mathcal{V}$, there exists a strongly static curve \[(x(\cdot),u(\cdot)):\mathbb{R}\to M\times\mathbb{R}\] such that $x(0)=y$. Since $\cV\subseteq \mathcal{A}$, then $u(t)=u_-(x(t))$ for all $t\in \R$. In particular, $u(0)=u_-(y)$. It follows from Proposition \ref{lem3.1qq} that
\begin{align*}
u_-(y)=u(0)=\limsup_{s\to +\infty}h^{x(0),u(0)}(x(0),s)=\limsup_{s\to +\infty}h^{y,u_-(y)}(y,s)=m(y,y),
\end{align*}
which together with $m(\cdot,y)\in \cS_+$ implies \[m(x,y)\geq \inf_{\substack{v_+(y)=u_-(y)\\  v_+\in \cS_+ } }v_+(x).\]

On the other hand, we have
\begin{align*}
v_+(x)&=T_t^+v_+(x)\geq \sup_{t>0}h^{y,v_+(y)}(x,t)\\
&=\sup_{t>0}h^{y,u_-(y)}(x,t)\geq \limsup_{t\to +\infty}h^{y,u_-(y)}(x,t)=m(x,y).
\end{align*}
It means \[\inf_{\substack{v_+(y)=u_-(y)\\  v_+\in \cS_+ } }v_+(x)\geq m(x,y).\]
This completes the proof of Lemma \ref{reminx}.
\end{proof}

\begin{remark}\label{anottr}
In general contact cases,
\begin{itemize}
\item [(1)] $m(x,y)$ is Lipschitz continuous in $x$, but it may not be continuous in $y$;
\item [(2)] the equality (\ref{wl5}) may not hold for $y\in M\backslash \mathcal{V}$.
\end{itemize}
We still consider the  Hamilton-Jacobi equation in Proposition \ref{exppp}:
\begin{equation}\label{chjjk}
\lambda u+\frac{1}{2}|Du|^2+Du\cdot V(x)=0,\quad\ x\in \mathbb{T},
\end{equation}
where $0<\lambda<|V'(x_2)|$.
By definition, for $y_0\in \mathbb{T}$,
\[m(y_0,y_0)=-\liminf_{\tau\to+\infty}\inf_{\gm(0)=\gm(\tau)=y_0}\int_0^\tau e^{\lambda s} \frac{1}{2}|\dot{\gm}(s)-V(\gm(s))|^2ds.\]
By Lemma \ref{hyxuxttttt}, $m(\cdot,y_0)\in \cS_+$. We choose a point $y_0\neq x_1,x_2$. It is not difficult to verify $m(y_0,y_0)<0$.
It follows that $m(x,y_0)=w_+(x)$ for all $x\in \mathbb{T}$. On the other hand, by Lemma \ref{lem3.1qq}(2), $m(x_2,x_2)=0$. Thus, $m(x,x_2)=u_+(x)\equiv 0$ for all $x\in \mathbb{T}$. Then
\[\lim_{\substack{y\to x_2\\  y\neq x_2 } }m(y_0,y)=w_+(y_0)<0=m(y_0,x_2),\]
which means $m(x,y)$ is not continuous at $y=x_2$. More precisely, for each $x\in \mathbb{T}$, $m(x,y)$ is continuous at $y\neq x_2$ and it is upper semicontinuous at $y=x_2$. This verifies Item (1).

For Item (2), we already know that if $y_0\neq x_1,x_2$, then $m(x,y_0)=w_+(x)$ for all $x\in \mathbb{T}$. Then
\[\inf_{\substack{v_+(y_0)=u_-(y_0)\\  v_+\in \cS_+ } }v_+(y_0)=u_+(y_0)= 0>w_+(y_0)=m(y_0,y_0).\]
Thus,  the equality (\ref{wl5}) does not hold.
\end{remark}

%\begin{remark}
%By \cite[Lemma 4.8]{WWY2}, a curve $(x(\cdot),u(\cdot)):\R\ri M\times\mathbb{R}$ is static if and only if $u(t)=u_-(x(t))$ for all $t\in\R$. Comparably, in view of Lemma \ref{lem3.1qq}(2) and Lemma \ref{reminx}, a curve $(x(\cdot),u(\cdot)):\R\ri M\times\mathbb{R}$  is strongly static if and only if
%\[u_-(x(t))=\inf_{v_+(x(s))=u_-(x(s))}v_+(x(t)),\quad \forall s,t\in \R.\]
%\end{remark}

\noindent{\it Proof of Theorem \ref{pooo}(2).}
By the definition of the Mather set, we have
\[\emptyset\neq \tilde{\mathcal{M}}_{u^*}:=\tilde{\mathcal{M}}\cap G_{u^*}\subseteq \tilde{\mathcal{I}}_{u^*}.\]
Based on Section \ref{mathre}, the recurrent points are dense in $\tilde{\mathcal{M}}_{u^*}$. Let $(x_0,p_0,u_0)\in \tilde{\mathcal{M}}_{u^*}$ be a recurrent point. According to Section \ref{recstr}, $(x_0,p_0,u_0)\in \tilde{\mathcal{V}}$. Then one can  choose
\begin{equation}\label{mvi}
x_0\in \mathcal{M}\cap \mathcal{V}\cap {\mathcal{I}}_{u^*},
\end{equation}
such that
\[\inf_{\substack{v_+(x_0)=u_-(x_0)\\  v_+\in \cS_+ } }v_+(x)=m(x,x_0),\quad \forall x\in M.\]
It remains to prove
\[\inf_{\mathcal{Z}_{\max}}v_+(x)=\inf_{\substack{v_+(x_0)=u_-(x_0)\\  v_+\in \cS_+ } }v_+(x).\]
By (\ref{mvi}),  $u^*(x_0)=u_-(x_0)$. It follows that for all $v_+\in \mathcal{Z}_{\max}$, $v_+(x_0)=u_-(x_0)$. Then
\[\inf_{\mathcal{Z}_{\max}}v_+(x)\geq\inf_{\substack{v_+(x_0)=u_-(x_0)\\  v_+\in \cS_+ } }v_+(x).\]
On the other hand, by Lemma \ref{hyxuxttttt} and Lemma \ref{reminx},
\[u(x):=\inf_{\substack{v_+(x_0)=u_-(x_0)\\  v_+\in \cS_+ } }v_+(x)\in \cS_+.\]
Due to the maximality of the set $\mathcal{Z}_{\max}$, we have $u\in \mathcal{Z}_{\max}$. It implies
\[\inf_{\mathcal{Z}_{\max}}v_+(x)\leq \inf_{\substack{v_+(x_0)=u_-(x_0)\\  v_+\in \cS_+ } }v_+(x).\]
This completes the proof of Theorem \ref{pooo}(2).
\EEnd
\subsection{Existence of transitive orbits}
We prove Theorem \ref{CONC} in this part.
By Lemma \ref{CONCLL1}, we only need to consider the case with $x_2\in \mathcal{V}$.  By Proposition \ref{hyxuxttttt} and Proposition \ref{reminx}, the function
 \[m(\cdot,x_2)=\limsup_{t\to+\infty}h^{x_2,u_-(x_2)}(\cdot,t)\]
is the minimal forward weak KAM solution of (\ref{hj}) equaling to $u_-(x_2)$ at $x_2$. By assumption, for each $v_+\in \cS_+$, $v_+(x_2)=u_-(x_2)$ implies $v_+(x_1)=u_-(x_1)$. Then
\begin{equation}\label{himp}
\limsup_{t\to+\infty}h^{x_2,u_-(x_2)}(x_1,t)=u_-(x_1).
\end{equation}
Note that for each $(x_0,u_0)\in M\times\R$,
\[\lim_{t\to+\infty}h_{x_0,u_0}(x,t)=u_-(x).\]
It yields
\[\lim_{t\to+\infty}h_{x_1,u_1}(x_2,t)=u_-(x_2)\]
Note that $u_-(x_1)=u_1$, $u_-(x_2)=u_2$.
Theorem \ref{CONC} follows from Lemma \ref{CONC1}.
\EEnd

%%%%%%%%%%%%%%%%%%%%%%%%%%%%%%%%%%%%%%%%%%%%%%%%%%%%%%%%%%%%%%%%%%%%%%

%%%%%%%%%%%%%%%%%%%%%%%%%%%%%%%%%%%%%%%%%%%%%%%%%%%%%%%
%%% Acknowledgements. ÖÂÐ»
%%%%%%%%%%%%%%%%%%%%%%%%%%%%%%%%%%%%%%%%%%%%%%%%%%%%%%%
\noindent {\bf Acknowledgements:} The authors would like to thank Prof. J. Yan for many helpful and constructive discussions throughout this paper, especially on the structure of the set of forward weak KAM solutions in the strictly increasing cases. The authors also warmly appreciate Dr. K. Zhao for showing us the possibility of the case in Proposition \ref{exppp}(iv). We also would like to thank the referees for their careful reading of the paper and invaluable
comments which are very helpful in improving this paper.
Lin Wang is supported by NSFC Grant No.  12122109.

\begin{appendix}

\section{Auxiliary results}\label{apC}
For the sake of generality, we will prove all of the results in this appendix under  (H1), (H2) and   $|\frac{\partial H}{\partial u}|\leq \lambda$ instead of (H3).
\subsection{Strong staticity and one-sided semi-staticity}\label{pui}
In this part, we prove Proposition \ref{lem3.2x}. First of all, we provide a way to construct ``long" minimizers from ``short" ones, which is a direct consequence of the Markov and monotonicity properties of the action functions.
\begin{proposition}\label{Mar-new}
	{Given any $x$, $y$ and $z\in M$, $u_1$, $u_2$ and $u_3\in\mathbb{R}$, $t$, $s>0$, let
	\[
	h_{x,u_1}(y,t)=u_2,\quad h_{y,u_2}(z,s)=h_{x,u_1}(z,t+s)=u_3
	\]
    \[
	(resp.\quad h^{z,u_3}(y,s)=u_2,\quad h^{y,u_2}(x,t)=h^{z,u_3}(x,t+s)=u_1).
	\]
Let $\gamma_1:[0,t]\to M$ be a minimizer of $h_{x,u_1}(y,t)$ (resp. $h^{y,u_2}(x,t)$) and $\gamma_2:[0,s]\to M$ be a minimizer of $h_{y,u_2}(z,s)$ (resp. $h^{z,u_3}(y,s)$).
	Then
	\[
    \gamma(\sigma):=\left\{\begin{array}{ll}
    \gamma_1(\sigma),\ \ \qquad\sigma\in[0,t],\\
   \gamma_2(\sigma-t),\,\quad\sigma\in[t,t+s],
    \end{array}\right.
	\]
	is a minimizer of $h_{x,u_1}(z,t+s)$ (resp. $h^{z,u_3}(x,t+s)$).}
\end{proposition}

\begin{proof} {We only prove the property for $h_{\cdot,\cdot}(\cdot,t)$ here, the proof for $h^{\cdot,\cdot}(\cdot,t)$ is similar. By the Markov property, we have
	\[
	h_{x,u_1}(\gamma(\sigma),\sigma)\leq h_{y,h_{x,u_1}(y,t)}(\gamma(\sigma),\sigma-t)= h_{y,u_2}(\gamma(\sigma),\sigma-t),\quad \forall \sigma\in[t,t+s].
	\]
	We assert that the above inequality is in fact an equality, i.e.,
	 \begin{align}\label{2-1}
	 h_{x,u_1}(\gamma(\sigma),\sigma)= h_{y,u_2}(\gamma(\sigma),\sigma-t),\quad \forall \sigma\in[t,t+s].
	 \end{align}
	 If the assertion is true, a direct calculation shows
%\begin{align*}
%	 &u_1+\int_0^{t+s}L(\gamma(\sigma),\dot{\gamma}(\sigma),h_{x,u_1}(\gamma(\sigma),\sigma))d\sigma\\
%=& u_1+\int_0^tL(\gamma_1(\sigma),\sigma),\dot{\gamma}_1(\sigma),h_{x,u_1}(\gamma_1(\sigma))d\sigma+\int_t^{t+s}L(\gamma(\sigma),\dot{\gamma}(\sigma),h_{x,u_1}(\gamma(\sigma),\sigma))d\sigma\\
%	 =&h_{x,u_1}(y,t)+\int_t^{t+s}L(\gamma(\sigma),\dot{\gamma}(\sigma),h_{y,u_2}(\gamma(\sigma),\sigma-t))d\sigma\\
%	 =&u_2+\int_0^sL(\gamma_2(\tau),\dot{\gamma}_2(\tau),h_{y,u_2}(\gamma_2(\tau),\tau))d\tau\\
%	 =&h_{y,u_2}(z,s)\\
%	 =&h_{x,u_1}(z,t+s),
%	 \end{align*}
	 \begin{align*}
	 u_1+\int_0^{t+s}L(\gamma(\sigma),\dot{\gamma}(\sigma),h_{x,u_1}(\gamma(\sigma),\sigma))d\sigma=h_{x,u_1}(z,t+s),
	 \end{align*}
	 which implies that $\gamma$ is a minimizer of $h_{x,u_1}(z,t+s)$.
	
	 Therefore, we only need to show (\ref{2-1}) holds. By contradiction, we assume there exists $t_0\in[t,t+s)$ such that
	 \[
	  h_{x,u_1}(\gamma(t_0),t_0)< h_{y,u_2}(\gamma(t_0),t_0-t).
	  \]
	 From the Markov and the monotonicity properties, we get
	 \begin{align*}
	 u_3&=h_{x,u_1}(z,t+s)\leq h_{\gamma(t_0),h_{x,u_1}(\gamma(t_0),t_0)}(z,t+s-t_0)\\
&<h_{\gamma(t_0),h_{y,u_2}(\gamma(t_0),t_0-t)}(z,t+s-t_0)\\
&=h_{\gamma_2(t_0),h_{y,u_2}(\gamma_2(t_0),t_0-t)}(z,t+s-t_0)\\
&=h_{y,u_2}(z,s)=u_3,
	 \end{align*}
	 which is a contradiction.}
\end{proof}

%Similarly, we have

%\begin{proposition}\label{Mar-newqq}
	%Given any $x$, $y$ and $z\in M$, $u_1$, $u_2$ and $u_3\in\mathbb{R}$, $t$, $s>0$, let
	%\[
	%h^{z,u_3}(y,s)=u_2,\quad h^{y,u_2}(x,t)=h^{z,u_3}(x,t+s)=u_1,
	%\]
%Let $\gamma_1:[0,t]\to M$ be a minimizer of $h^{y,u_2}(x,t)$ and $\gamma_2:[0,s]\to M$ be a minimizer of $h^{z,u_3}(y,s)$.
	%Then
	%\[
    %\gamma(\sigma):=\left\{\begin{array}{ll}
    %\gamma_1(\sigma),\ \ \qquad\sigma\in[0,t],\\
   %\gamma_2(\sigma-t),\,\quad\sigma\in[t,t+s],
    %\end{array}\right.
	%\]
	%is a minimizer of $h^{z,u_3}(x,t+s)$.
%\end{proposition}

%Let $\tilde{\mathcal{V}}$ be the set of $(x,p,u)\in T^*M\times\R$, for which there exists a strongly static orbit \[(x(\cdot),p(\cdot),u(\cdot)):\mathbb{R}\to T^*M\times\mathbb{R}\] such that
%\[(x(0),p(0),u(0))=(x,p,u).\]
%Let ${\mathcal{V}}=\pi \tilde{\mathcal{V}}$.
%By the definition of the strongly static set, $\tilde{\mathcal{S}}_s=\bar{\tilde{\mathcal{V}}}$, $\mathcal{S}_s=\bar{\mathcal{V}}$.

\noindent{\it Proof of {Proposition} \ref{lem3.2x}.}
We only need to prove if $(x,p_0,u)\in {\tilde{\mathcal{V}}}$, $(x,p_+,u)\in \tilde{\N}^+$,   then $p_0=p_+$. The other case is similar. For each $t\in \R$, let \[(x_1(t),p_1(t),u_1(t))=\Phi_t(x,p_0,u).\]For each $t\geq 0$, let  \[(x_2(t),p_2(t),u_2(t))=\Phi_t(x,p_+,u).\]
We need to prove if a globally minimizing curve $(x_1(\cdot),u_1(\cdot)):\mathbb{R}\to M\times\mathbb{R}$ satisfies for each $t_1, t_2\in\mathbb{R}$,
	\begin{equation}\label{appxx}
	u_1(t_2)=\inf_{s>0}h_{x_1(t_1),u_1(t_1)}(x_1(t_2),s),
\end{equation}
then $p_0=p_+$.

Since  $(x,p_+,u)\in \tilde{\N}^+$, then $(x_2(\cdot),u_2(\cdot)):\R_+\to M\times\R$ is positively semi-static.
Fixing $\delta>0$, by the Markov property, we have
\[h_{x_1(-\delta),u_1(-\delta)}(x_2(\delta),2\delta)=\inf_{y\in M}h_{y,h_{x_1(-\delta),u_1(-\delta)}(y,\delta)}(x_2(\delta),\delta).\]
Note that
\[h_{x_1(-\delta),u_1(-\delta)}(x,\delta)=u,\quad  h_{x,u}(x_2(\delta),\delta)=u_2(\delta).\]
	It follows that
\[h_{x_1(-\delta),u_1(-\delta)}(x_2(\delta),2\delta)\leq h_{x,u}(x_2(\delta),\delta).\]
We assert that the inequality above is indeed an equality. If the assertion is true, then by Proposition \ref{Mar-new}, the curve defined by	
	\[
	\gamma(\sigma):=\left\{\begin{array}{ll}
	x_1(\sigma-\delta), \quad\sigma\in[0,\delta],\\
	x_2(\sigma-\delta),\,\quad\sigma\in[\delta,2\delta],
	\end{array}\right.
	\]
	is a minimizer of $h_{x_1(-\delta),u_1(-\delta)}(x_2(\delta),2\delta)$ and it is of class $C^1$. Thus,
\[p_0=\frac{\partial L}{\partial \dot{x}}(x,\dot{\gm}(0),0)=p_+.\]

It remains to verify the assertion. By contradiction, we assume that  there exists $\Delta>0$ such that	
	\[
	h_{x_1(-\delta),u_1(-\delta)}(x_2(\delta),2\delta)=h_{x,u}(x_2(\delta),\delta)-\Delta.
	\]

By (\ref{appxx}),  for each $\eps>0$, one can find $s_0>0$ such that
\[|h_{x,u}(x_1(-\delta),s_0)-u_1(-\delta)|\leq \eps.\]
From the Lipschitz continuity of $h_{x_0,u_0}(x,t)$ w.r.t. $u_0$,
\[|h_{x_1(-\delta),h_{x,u}(x_1(-\delta),s_0)}(x_2(\delta),2\delta)-h_{x_1(-\delta),u_1(-\delta)}(x_2(\delta),2\delta)|\leq k\eps,\]
where $k$ denotes the Lipschitz constant of $h_{x_0,u_0}(x,t)$ w.r.t. $u_0$.
It follows from the definition of $\tilde{\N}^+$ that
\begin{align*}
u_2(\delta)&=\inf_{\tau>0}h_{x,u}(x_2(\delta),\tau),\\
&\leq h_{x,u}(x_2(\delta),s_0+2\delta),\\
&\leq h_{x_1(-\delta),h_{x,u}(x_1(-\delta),s_0)}(x_2(\delta),2\delta),\\
&\leq h_{x_1(-\delta),u_1(-\delta)}(x_2(\delta),2\delta)+k\eps,\\
&=h_{x,u}(x_2(\delta),\delta)-\Delta+k\eps.
\end{align*}
Note that $\Delta$, $k$ are constants independent of $\eps$. Taking $\eps$ small enough, we have
	\[
	u_2(\delta)\leq h_{x,u}(x_2(\delta),\delta)-\frac{\Delta}{2}=u_2(\delta)-\frac{\Delta}{2},
	\]
	which is a contradiction. This completes the proof of {Proposition} \ref{lem3.2x}.

\vspace{1ex}

\subsection{Lipschitz continuity of Ma\~{n}\'{e} potentials}\label{apB}

Let $(x(\cdot),u(\cdot)):\R\ri M\times\R$ be a semi-static curve. Fixing $\tau\in\R$, we consider two kinds of the Ma\~{n}\'{e} potentials as follows:
\[\check{K}_{\tau}(x):=\inf_{s>0}h_{x(\tau),u(\tau)}(x,s),\quad \hat{K}_{\tau}(x):=\sup_{s>0}h^{x(\tau),u(\tau)}(x,s).\]

In this part, we will prove
\begin{proposition}\label{immm}
Given $\tau\in \R$, let $U$ be an open set containing $x(\tau)$. Then both $\check{K}_{\tau}(x)$ and $\hat{K}_{\tau}(x)$ are uniformly Lipschitz continuous with respect to $x\in M\backslash U$.
\end{proposition}
%Under the assumption (H1)-(H3), $h_{x(\tau),u(\tau)}(x,s)$ is non-expensive with respect to $s$ (see Proposition \ref{relation}), from which the Lipschitz continuity of $x\mapsto\check{K}_{\tau}(\cdot)$ was shown in Lemma \ref{uueequi}. Nevertheless, it is more difficult to establish the same regularity  of $x\mapsto\hat{K}_{\tau}(\cdot)$ due to the lack of non-expansiveness of $h^{x(\tau),u(\tau)}(x,s)$ with respect to $s$. In order to achieve that, we have to prove Proposition \ref{immm} under the assumption (H1), (H2) and   $|\frac{\partial H}{\partial u}|\leq \lambda$ instead of (H3).

 We only need to prove this proposition for $\check{K}_{\tau}(x)$, from which the Lipschitz continuity of $\hat{K}_{\tau}(x)$   can be obtained by a similar way.
%\begin{proposition}\label{hupin}
%For each $(x_0,u_0,x,t)\in M\times\R\times M\times (0,+\infty)$,
%\begin{align}
%\bar{h}_{x_0,u_0}(x,t)=-h^{x_0,-u_0}(x,t),\quad \bar{h}^{x_0,u_0}(x,t)=-h_{x_0,-u_0}(x,t),
%\end{align}
%where $\bar{h}_{x_0,u_0}(x,t)$ and $\bar{h}^{x_0,u_0}(x,t)$ denote the forward and backward action functions with respect to $\bar{H}(x,p,u):=H(x,-p,-u)$ respectively.
%\end{proposition}
%See \cite[Proposition 3]{WWY3} for a proof Proposition \ref{hupin}.

\begin{lemma}\label{fu1}
	Let  $\mathcal{K}$ be a compact subset of $M$ and $u_0\in\mathbb{R}$. Then for any $x_0\in M\backslash \mathcal{K}$, we have $\lim_{t\to 0^+}h_{x_0,u_0}(x,t)=+\infty$ uniformly in $x\in\mathcal{K}$.
\end{lemma}

\begin{proof}
Given $(x,t)\in \mathcal{K}\times(0,+\infty)$, let  $\Gamma^t_{x_0,x}$ be the set of the minimizers  of $h_{x_0,u_0}(x,t)$.  Namely, for each $\gm:[0,t]\to M$ contained in $\Gamma^t_{x_0,x}$, we have $\gm(0)=x_0$, $\gm(t)=x$ and
\begin{equation}\label{hchu}
h_{x_0,u_0}(x,t)=u_0+\int_0^tL\big(\gamma(\tau),\dot{\gamma}(\tau),h_{x_0,u_0}(\gamma(\tau),\tau)\big)d\tau.
\end{equation}
Let
	\[
	g_{x}(t):=\inf_{\gm\in \Gamma^t_{x_0,x}}\sup_{0\leq s\leq t}h_{x_0,u_0}(\gamma(s),s).
	\]
	{We split} the remaining proof by two steps.

	\medskip
\noindent  {\bf Step 1}: we show $\lim_{t\to 0^+}g_{x}(t)=+\infty$, uniformly in $x\in \mathcal{K}$.\\[2mm]
	By contradiction, we assume there exist   $x_n\in\mathcal{K}$ and  $\gamma_n\in \Gamma^{t_n}_{x_0,x_n}$ with $t_n\to 0$ as $n\to +\infty$ such that
	\begin{align}\label{2-600}
	h_{x_0,u_0}(\gamma_n(s),s)<C_1,\quad \forall s\in[0,t_n],
	\end{align}
where $C_1$ is a constant independent of $n$.
Let $A:=\inf_{(x,\dot{x})\in TM}L(x,\dot{x},u_0)$. Given $T_0>0$,
let \[C_2:=|u_0|e^{\lambda T_0}+\frac{|A-\lambda u_0|}{\lambda}\left(e^{\lambda T_0}-1\right)+1,\] where $\lambda$ is a Lipschitz constant of $H(x,p,u)$ w.r.t. $u$.

\bigskip
\noindent \textbf{Claim.} For any $(x,t)\in M\times (0,T_0]$, $h_{x_0,u_0}(x,t)>-C_2$.
\bigskip

\noindent {\it Proof of the claim.}
We assume by contradiction that there exists $(x_1,t_1)\in M\times (0,T_0]$ such that $h_{x_0,u_0}(x_1,t_1)\leq -C_2$. Let $\gamma\in \Gamma^{t_1}_{x_0,x_1}$. Denote
$u(s):=h_{x_0,u_0}(\gamma(s),s)$ for $s\in [0,t_1]$. Note that $u_0\geq -C_2$. Since $u(s)$ is continuous on $(0,t_1]$, and $\lim_{s\rightarrow0^+}u(s)=u_0$, there exists a closed interval $[{s}_1,{s}_2]\subseteq[0,t_1]$ such that
\[
u({s}_1)=u_0,\quad u({s}_2)=-C_2,\quad -C_2\leq u(s)\leq u_0,\quad \forall s\in[{s}_1,{s}_2].
\]
Since $\gamma$ satisfies (\ref{hchu}), based on the variational principle (see Proposition \ref{IVP}),
\begin{equation*}
  \dot{u}(s)=L(\gamma(s),\dot{\gamma}(s),u(s))\geq A+\lambda(u(s)-u_0),\ \ \forall s\in[{s}_1,{s}_2].
\end{equation*}
A direct calculation yields for any $s\in[{s}_1,{s}_2]$,
\begin{align*}
u(s)&\geq u_0 e^{\lambda(s-{s}_1)}+\frac{A-\lambda u_0}{\lambda}\left(e^{\lambda(s-{s}_1)}-1\right)\\
&\geq-|u_0|e^{\lambda T_0}-\frac{|A-\lambda u_0|}{\lambda}\left(e^{\lambda T_0}-1\right)\\
&>-C_2.
\end{align*}
This contradicts $u({s}_2)=-C_2$. Then the claim is true.

\bigskip

For $n$ large enough, we have $t_n<T_0$. Let $C:=\max\{C_1,C_2\}$. Based on (\ref{2-600}) and the assertion above,
\begin{equation}\label{hccc}
|h_{x_0,u_0}(\gm_n(s),s)|\leq C,\quad \forall s\in [0,t_n].
\end{equation}

	Let $\delta:=\text{dist}(x_0,\mathcal{K})$, where dist$(\cdot,\cdot)$ denotes a distance induced by the Riemannian metric on $M$. Let
 \[B:=\frac{C+1+|u_0|}{\delta}.\]
Since $L(x,\dot{x},0)$ is superlinear in $\dot{x}$, then there is $D:=D(B)\in\mathbb{R}$ such that $L(x,\dot{x},0)\geq B\|\dot{x}\|_x-D$ for all $(x,\dot{x})\in TM$. Since $t_n\to 0^+$ as $n\to +\infty$, for $n$ large enough, we get $|(D+\lambda C)t_n|<1$. Note that
	\begin{align*}
	h_{x_0,u_0}(x_n,t_n)&=u_0+\int_0^{t_n}L(\gamma_n(s),\dot{\gamma}_n(s),h_{x_0,u_0}(\gamma_n(s),s))ds\\
	&\geq u_0+\int_0^{t_n}L(\gamma_n(s),\dot{\gamma}_n(s),0)ds-\lambda\int_0^{t_n}|h_{x_0,u_0}(\gm_n(s),s)|ds\\
	&\geq u_0+B\delta-Dt_n-\lambda C t_n\\
	&= u_0+B\delta-(D+\lambda C)t_n\\
	&>C,
	\end{align*}
	which contradicts (\ref{hccc}). Therefore, $\lim_{t\to 0^+}g_{x}(t)=+\infty$, uniformly in $x\in \mathcal{K}$.

\medskip
\noindent  {\bf Step 2}:
	 we show $\lim_{t\to 0^+}h_{x_0,u_0}(x,t)=+\infty$, uniformly for all $x\in\mathcal{K}$.\\[2mm]
	From Step 1, for any $N>0$, there is $t_N>0$ such that $g_{x}(t)>N$ for $t<t_N$ and all $x\in\mathcal{K}$.  Let $\gm\in \Gamma_{x_0,x}^t$. Note that $h_{x_0,u_0}(\gamma(s),s)\to u_0$ as $s\to 0^+$. One can find $s_0\in [0,t]$ such that $h_{x_0,u_0}(\gamma(s_0),s_0)=N$.

Note that
	\begin{align*}
	h_{x_0,u_0}(x,t)=h_{x_0,u_0}(\gamma(s_0),s_0)+\int_{s_0}^tL(\gamma(s),\dot{\gamma}(s),h_{x_0,u_0}(\gamma(s),s))ds.
	\end{align*} Similar to the argument above, for $t<t_N$, we have
	\begin{align*}
	h_{x_0,u_0}(x,t)&\geq N+\int_{s_0}^tL(\gamma(s),\dot{\gamma}(s),0)ds-\lambda\int_{s_0}^t|h_{x_0,u_0}(\gamma(s),s)|ds\\
	&\geq N+B\delta-(D+\lambda C)(t-s_0),
	\end{align*}
	Let $t\to 0^+$. Then $t-s_0\to 0^+$. Moreover,
	$h_{x_0,u_0}(x,t)>N$ for each $x\in \mathcal{K}$, which completes the proof.	
\end{proof}

By \cite[Lemma C.1]{WWY4}, we have
\begin{lemma}\label{uueequi}
Let $(x(\cdot),u(\cdot)):\R\ri M\times\mathbb{R}$ be a semi-static curve. Then for each $\delta>0$,
\begin{itemize}
\item \text{Uniform Boundedness}:   there exists a constant $K>0$ independent of $t$ such that for $t>\delta$ and each $x\in M$, $s\in \R$, \[|h_{x(s),u(s)}(x,t)|\leq K,\quad|h^{x(s),u(s)}(x,t)|\leq K;\]
\item \text{Equi-Lipschitz Continuity}:  there exists a constant $\kappa>0$ independent of $t$ such that for $t> 2\delta$ and $s\in \R$, both $x\mapsto h_{x(s),u(s)}(x,t)$  and $x\mapsto h^{x(s),u(s)}(x,t)$ are $\kappa$-Lipschitz continuous on $M$.
\end{itemize}
\end{lemma}

\medskip
\noindent{\it Proof of Proposition \ref{immm}.}  We only need to prove this proposition for $\check{K}_{\tau}(x)$. Let $U$ be an open neighborhood of $x(\tau)$ and $x\in \mathcal K:=M\backslash U$. By Lemma \ref{fu1}, we have $\lim_{t\to 0^+}h_{x(\tau),u(\tau)}(x,t)=+\infty$ uniformly for $x\in\mathcal K$. Thus, there exists $\delta>0$ independent of $x\in\mathcal K$ such that
	\[
	\check{K}_{\tau}(x):=\inf_{s>0}h_{x(\tau),u(\tau)}(x,s)=\inf_{s>\delta}h_{x(\tau),u(\tau)}(x,s),\quad \forall x\in \mathcal{K}.
	\]
It follows from Lemma \ref{uueequi} that
\begin{align*}
&\left|\check{K}_{\tau}(x)-\check{K}_{\tau}(y)\right|\\
=&\left|\inf_{s>\delta}h_{x(\tau),u(\tau)}(x,s)-\inf_{s>\delta}h_{x(\tau),u(\tau)}(y,s)\right|\\
\leq &\sup_{s>\delta}\left|h_{x(\tau),u(\tau)}(x,s)-h_{x(\tau),u(\tau)}(y,s)\right|\\
\leq &\kappa\ d(x,y).
\end{align*}
This completes the proof.
\EEnd

\section{Proof of Proposition \ref{exppp}}\label{proofexppp}

It is clear that $u_-\equiv 0$ is the unique element in $\cS_-$, and $u_+\equiv 0\in \cS_+$.
\subsection{On Item (i)}
{The contact Hamilton equations read}
\begin{align}\label{xjja}
\left\{
        \begin{array}{l}
        \dot{x}=p+V(x),\\
        \dot{p}=-pV'(x)-\lambda p,\\
        \dot{u}=p(p+V(x))-H(x,p,u).
         \end{array}
         \right.
\end{align}
Denote the solution of (\ref{xjja}) by $(x(t),p(t),u(t))$.

A direct calculation shows
\[\frac{dH}{dt}=-\lambda H(x(t),p(t),u(t)).\]
Thus, we only need to consider the dynamics on zero energy level set
\[E:=\{(x,p,u)\in T^*\mathbb{T}\times\R\ |\ H(x,p,u)=0\}.\]
To verify Item (i), it suffices to consider the linearization of (\ref{xjja}) in a neighborhood of $(x_2,0)\in T^*\mathbb{T}$. It is formulated as
\begin{align}\label{limaxt}
\left[
        \begin{array}{c}
        \dot{x}\\
        \dot{p}\\
         \end{array}
         \right]=\left[
        \begin{array}{cc}
        V'(x_2)& 1\\
        0& -(V'(x_2)+\lambda)\\
         \end{array}
         \right]
       \left[ \begin{array}{c}
        x-x_2\\
        p\\
         \end{array}
         \right].
\end{align}
By the assumptions on $V(x)$, Item (i) holds.

\subsection{On Item (ii)}
This item was proved by \cite[Proposition 1]{WWY4}. We omit it.

\subsection{On Item (iii)}
To fix the notations, we use $\mathcal{D}_{v_+}$ to denote the set of differentiable points of $v_+$. For each $v_+\in \cS_+$,  we know that it is semiconvex with linear modulus (see \cite[Theorem 5.3.6]{CS}). Moreover, $\mathcal{D}_{v_+}$ has full Lebesgue measure on $\T$.
  Denote $\check{\M}:=\Pi^*\tilde{\mathcal{M}}$, where $\Pi^*:T^*\T\times\R\to T^*\T$. Namely, $\check{\M}$ denotes the projection of the Mather set $\tilde{\mathcal{M}}$ to $T^*\T$. Let $\check\Phi_t:=\Pi^*\Phi_t$.

  The following lemma is from \cite[Proposition 4.5]{dw}. In \cite{dw}, the Hamiltonian is required to be of class $C^3$ since that work is based on the Aubry-Mather theory for contact Hamiltonian systems developed in \cite{WWY2}. In \cite{WWY2}, $C^3$ is assumed for a technical reason. By \cite{MS}, the requirement of $C^3$ regularity can be reduced to  $C^2$  since this lemma only involves Aubry-Mather theory for discounted systems.
\begin{lemma}\label{locstab}
Let us consider
\begin{equation}\label{last HJ eq}
 \lambda u+\check H(x,d_x u)=c(\check H)\qquad\hbox{in $\T$},
\end{equation}
where $\check H:T^*\T\to\R$ is a $C^2$--Hamiltonian, satisfying Tonelli assumptions and $\T$ is a flat circle.
 Let   $(x_0,0)\in \check{\M}$ be
 a saddle point for the discounted flow generated by
  \begin{align}\label{cott}\tag{DH}
        \begin{cases}
        \dot{x}=\frac{\partial { \check H  }}{\partial p}(x,p),\smallskip\\
        \dot{p}=-\frac{\partial { \check H  }}{\partial x}(x,p)- \lambda p.
         \end{cases}
\end{align}
Given $v_+\in \mathcal{S}_+$ with $v_+(x_0)=u_-(x_0)$, let $\bar{x}_1,\bar{x}_2$ be two differentiable points of $v_+$ with $x_0\in (\bar{x}_1,\bar{x}_2)$. Denote $\bar{p}_i:=d_{\bar{x}_i}v_+$ with $i=1,2$. If
\[(x_0,0)\in \omega(\bar{x}_1,\bar{p}_1)\cap \omega(\bar{x}_2,\bar{p}_2),\]
where $ \omega(\bar{x}_i,\bar{p}_i)$ denotes the $\omega$-limit set of $(\bar{x}_i,\bar{p}_i)$. Then there exists $\delta:=\delta(v_+)>0$ such that  $d_xv_+$ exists for all $x\in [x_0-\delta,x_0+\delta]$, and the set
\[\{(x,d_xv_+)\ |\ x\in [x_0-\delta,x_0+\delta]\}\]
coincides with the local stable submanifold of $(x_0,0)$.
\end{lemma}

%
%\begin{lemma}\label{locglgl}
%Let $v$ be a forward weak KAM solution of
%\begin{equation}\label{uf2d}
%\lambda u(x)+\frac{1}{2}|Du|^2+Du\cdot V(x)=0,\quad \lambda>0,\ x\in \T.
%\end{equation}
%Then $d_xv<0$ for all $x\in \mathcal{D}_v\cap(x_1,x_2)$, and $d_xv>0$ for all $x\in \mathcal{D}_v\cap\T\backslash(x_1,x_2)$.
%\end{lemma}
%
%\begin{proof}
%By (\ref{uf2d}), a direct calculation shows
%\begin{equation}\label{uf2d1}
%d_xv=-V(x)\pm\sqrt{V^2(x-2\lambda V(x))}.
%\end{equation}
%By assumptions, we have $V(x)>0$ for all $x\in (x_1,x_2)$ and $V(x)<0$ for all $x\in \T\backslash(x_1,x_2)$. Then we complete the proof.
%\end{proof}

Next, we prove Item (iii). Note that $u_-\equiv 0$ is the classical solution. Then  $u_+\equiv 0$ is the maximal forward weak KAM solution.
 By \cite[Proposition 10]{WWY4}, if the forward weak KAM solution is unique, then $\tilde{\cS}_s=\tilde{\A}$. It follows from Proposition \ref{exppp}(ii) that there exists $v_+\in \cS_+$  different from  $u_+$. Thus, $v_+\leq 0$ and there exists $x_0\in \mathbb{T}$ such that $v_+(x_0)<0$. Consider
 \[\mathcal{I}:=\{x\in \mathbb{T}\ |\ v_+(x)=0\}.\]
  Then $\mathcal{I}$ is a compact invariant set by $\pi^*\Phi_t$, where $\pi^*:T^*\mathbb T\times\R\ri \mathbb T$ denotes the standard projection. Denote a fundamental domain of $\mathbb{T}$ by $[x_1,x_1+1)$. Consequently, there are several possibilities for $\mathcal{I}$ restricting on $[x_1,x_1+1)$:
 \begin{equation}\label{iasett}
 \{x_1\},\ \{x_2\},\ \{x_1,x_2\},\ [x_1,x_2],\ [x_2,x_1+1)\cup \{x_1\},\ [x_1,x_1+1).
 \end{equation}

The remaining proof is divided into two steps.

\vspace{1ex}

\noindent{\bf Step 1.} For each $v_+\in \cS_+$, $v_+(x_2)<0$.

\noindent
We assert that if $v_+(x_2)=0$, then $v_+\equiv u_+=0$. In fact, by (\ref{iasett}), it suffices to show that there exists $\eps>0$, such that $v_+=0$ on $[x_2-\eps,x_2+\eps]$. By contradiction, we assume there exists $\bar{x}\in \mathcal{D}_{v_+}\cap [x_2-\frac{1}{2}\eps,x_2)$, such that $v_+(\bar{x})<0$. Let $\bar{p}:=d_{\bar{x}}v_+$. Let $(x(t),p(t)):=\check\Phi_t(\bar{x},\bar{p})$ for all $t\geq 0$. Then one can find $t_0\geq 0$ such that $x(t_0)\in [x_2-\eps, x_2)$ and $d_xv$ exists at $x=x(t_0)$ with $d_{x(t_0)}v_+>0$. Note that
\[\dot{x}(t_0)=p(t_0)+V(x(t_0))=d_{x(t_0)}v_++V(x(t_0)).\]
By the definition of $V(x)$, $V(x(t_0))>0$. It follows that $\dot{x}(t_0)>0$. Thus, $(x_2,0)\in \omega(\bar{x},\bar{p})$. By (\ref{limaxt}), if $\lambda<|V'(x_2)|$, $(x_2,0)$ is a saddle point in $\check{\M}$. By Lemma \ref{locstab},
$v_+=0$ on $[x_2-\eps,0]$. Similarly, we have  $v_+=0$ on $[x_2-\eps,x_2+\eps]$.

By the assertion above, we have $v_+(x_2)<0$ and $\mathcal{I}=\{x_1\}$.

\vspace{1ex}

\noindent{\bf Step 2.} The forward weak KAM solution $v_+$ different from $u_+\equiv  0$ is unique.

\noindent Let $\bar{x}\in \mathcal{D}_{v_+}\cap (x_1,x_2)$.  By Proposition \ref{exppp}(ii),
\begin{equation}\label{math44}
\tilde{\mathcal{M}}=\left\{(x_1,0,0),\ (x_2,0,0)\right\},
\end{equation}
which together with  $v_+(x_2)<0$ implies $(x_1,0)\in \omega(\bar{x},\bar{p})$. In fact, if $(x_1,0)\notin \omega(\bar{x},\bar{p})$, then $(x_2,0)\in \omega(\bar{x},\bar{p})$. Moreover, $v_+(x(t))\to v_+(x_2)<0$ as $t\to+\infty$, which contradicts (\ref{math44}).
By (\ref{limaxt}),  $(x_1,0)\in \check{\mathcal{M}}$ is a saddle point. By Lemma \ref{locstab}, for any $v_+\in \cS_+$, there exists $\delta:=\delta(v_+)>0$ such that  $d_xv_+$ exists for all $x\in [x_1-\delta,x_1+\delta]$, and the set
\[\{(x,d_xv_+)\ |\ x\in [x_1-\delta,x_1+\delta]\}\]
coincides with the local stable submanifold of $(x_1,0)$.

By contradiction, we assume there exists another $\bar{v}_+$ that is different from both $u_+$ and $v_+$. From the discussion above, we have $\bar{v}_+(x_1)=v_+(x_1)=0$ and   one can find ${\delta'}>0$ such that $d_x\bar{v}_+=d_xv_+$ on $[x_1-\delta',x_1+\delta']$.
That yields $\bar{v}_+={v}_+$ on $[x_1-\delta',x_1+\delta']$. By \cite[Proposition 4]{dw}, $\bar{v}_+=v_+$ on $\T$. Therefore, the forward weak KAM solution $v_+$ different from $u_+$ is unique.

 \subsection{On Item (iv)}

 To verify Item (iv), we only need to construct an example to show that for $\lambda$ large enough, $\cS_+$  contains more than two elements. We choose $\lambda=3$ and $V(x)=\sin x$. Then $x_1=0$, $x_2=\pi$. Let
 \[\varphi(x):=\cos x-1,\quad \psi(x):=-\cos x-1.\]
 A direct calculation shows that
 \[3\varphi+\frac{1}{2}|D\varphi|^2+D\varphi\cdot \sin x\leq 0,\quad 3\psi+\frac{1}{2}|D\psi|^2+D\psi\cdot \sin x\leq 0,\]
 which means both $\varphi$ and $\psi$ are viscosity subsolutions of
 \begin{equation}\label{uf2d}
3 u(x)+\frac{1}{2}|Du|^2+Du\cdot \sin(x)=0,\quad \ x\in \T.
\end{equation}
 Then
 \begin{equation}\label{monnnn4}
 T_t^+\varphi\leq \varphi,\quad T_t^+\psi\leq \psi.
 \end{equation}
 Note that $\varphi,\psi\leq 0$ and $\varphi(0)=\psi(\pi)=0=u_-(0)$. By \cite[Proposition 13]{WWY3}, both $T_t^+\varphi$ and $T_t^+\psi$ converge as $t\to+\infty$. Let
 \[u_1(x):=\lim_{t\to+\infty}T_t^+\varphi, \quad u_2(x):=\lim_{t\to+\infty}T_t^+\psi.\]
 Then $u_1,u_2\in \cS_+$. By (\ref{monnnn4}) and the constructions of $\varphi$ and $\psi$, we know that $u_1$, $u_2$ and $u_+\equiv 0$ are different from each other.

\end{appendix}

\end{document}